\documentclass[12pt,oneside,english]{amsart}
\usepackage[T1]{fontenc}
\usepackage[utf8]{inputenc}
\usepackage[a4paper]{geometry}
\usepackage{amstext}
\usepackage{amsthm}
\usepackage{amssymb}
\usepackage{graphicx}
\usepackage{hyperref}
\usepackage{setspace}

\makeatletter
\numberwithin{figure}{section}
\theoremstyle{plain}
\newtheorem{thm}{\protect\theoremname}
\theoremstyle{definition}
\newtheorem{defn}[thm]{\protect\definitionname}
\theoremstyle{plain}
\newtheorem{lem}[thm]{\protect\lemmaname}
\theoremstyle{remark}
\newtheorem{rem}[thm]{\protect\remarkname}
\theoremstyle{plain}
\newtheorem{prop}[thm]{\protect\propositionname}
\theoremstyle{plain}
\newtheorem{cor}[thm]{\protect\corollaryname}
\theoremstyle{definition}
\newtheorem{example}[thm]{\protect\examplename}

\makeatother

\usepackage{babel}
\providecommand{\corollaryname}{Corollary}
\providecommand{\definitionname}{Definition}
\providecommand{\examplename}{Example}
\providecommand{\lemmaname}{Lemma}
\providecommand{\propositionname}{Proposition}
\providecommand{\remarkname}{Remark}
\providecommand{\theoremname}{Theorem}

\begin{document}
~~\includegraphics[scale=0.52]{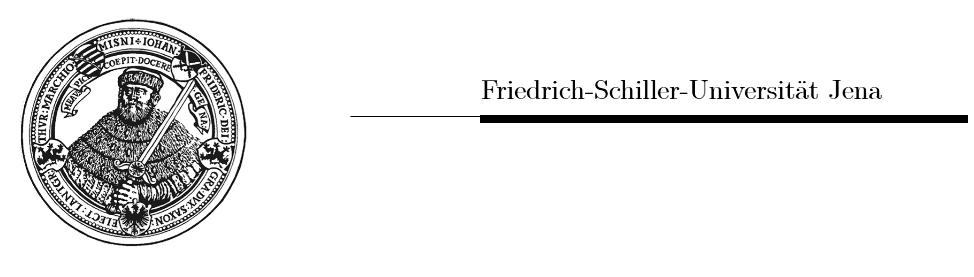}
\title{Embeddings of Weighted Morrey Spaces}
\maketitle
\begin{doublespace}
\begin{center}
\textbf{Master Thesis}
\par\end{center}

\begin{center}
to achieve the Academic Degree 
\par\end{center}

\begin{center}
Master of Science (M. Sc.)
\par\end{center}

\begin{center}
in Mathematics
\par\end{center}

\begin{center}
FRIEDRICH-SCHILLER-UNIVERSIT\"{A}T JENA
\par\end{center}

\begin{center}
Fakultät für Mathematik und Informatik\textbf{\medskip{}
\bigskip{}
}
\par\end{center}

\begin{center}
\textbf{\bigskip{}
}
\par\end{center}

\begin{center}
\textbf{\bigskip{}
}
\par\end{center}

\begin{center}
\textbf{\bigskip{}
}
\par\end{center}
\end{doublespace}

\begin{flushleft}
~~~~~~~~~~~~~~~~~~~~~~~~~~~~~~~~~~~~~Submitted
by: ~~Marcus Gerhold
\par\end{flushleft}

\begin{flushleft}
~~~~~~~~~~~~~~~~~~~~~~~~~~~~~~~~~~~~~Date
of Birth: ~~30/09/1989
\par\end{flushleft}

\begin{flushleft}
~~~~~~~~~~~~~~~~~~~~~~~~~~~~~~~~~~~~~~~~~~~~~Adviser:
~~Prof. Dr. Dorothee D. Haroske 
\par\end{flushleft}

\begin{doublespace}
\begin{center}
~~~~~~~~Jena, 30/10/2013
\par\end{center}
\end{doublespace}

\thispagestyle{empty}

\newpage{}
\begin{abstract}
In this master thesis we recall already established definitions and
basic properties of classical Morrey spaces in an attempt to expand
known facts to their weighted counterparts. To do so, we will recall
properties of Muckenhoupt weights, which we will use to derive further
properties of weighted Morrey spaces. We will also show the boundedness
of the Hardy-Littlewood maximal operator on weighted Morrey spaces.
Throughout this thesis we will have a look at example weights and
show in which case Morrey spaces equipped with these weights are embedded
in each other. 
\end{abstract}

\thispagestyle{empty}

\newpage{}

\tableofcontents{}

\newpage{}
\thanks{I would like to express my sincere gratitude and appreciation to my
adviser Prof. Dr. Dorothee Haroske for her constant support, guidance
and patience throughout my work on this thesis.}

\thispagestyle{empty}

\newpage{}

\part*{Mathematical Background}

We begin by recapturing some basic notations and conventions of which
will be assumed that the reader is well acquainted with. If this,
however, is not the case, we refer to \cite{H. Triebel} for the further
mathematical background needed within this thesis.

Let $n\in\mathbb{N}$. As usual let $L_{p}\left(\mathbb{R}^{n}\right):=\left\{ f:\Vert f\vert L_{p}\left(\mathbb{R}^{n}\right)\Vert<\infty\right\} $
denote the function spaces of equivalence classes $\left[f\right]$
of $p$ integrable functions, where $0<p\leq\infty$. We use the corresponding
quasi-norm $\Vert f\vert L_{p}\left(\mathbb{R}^{n}\right)\Vert:=\left(\int_{\mathbb{R}^{n}}\vert f\left(x\right)\vert^{p}dx\right)^{1/p}$
and the usual change for $p=\infty$, i.e. $\Vert f\vert L_{\infty}\left(\mathbb{R}^{n}\right)\Vert=\inf_{C>0}\left(\mu\left(\left\{ x\in\mathbb{R}^{n}:\left|f\left(x\right)\right|>C\right\} \right)=0\right)$
with $\mu$ being the Lebesgue measure. By $L_{p}^{loc}$ we refer
to $p$ integrable functions on every compact set $K$ and with $w-L_{p}\left(\mathbb{R}^{n}\right)$
we mean the weak Lebesgue spaces defined in the usual manner $w-L_{p}\left(\mathbb{R}^{n}\right)=\left\{ f:\Vert f\vert w-L_{p}\left(\mathbb{R}^{n}\right)\Vert<\infty\right\} $,
with their corresponding norm $\Vert f\vert w-L_{p}\left(\mathbb{R}^{n}\right)\Vert=\sup_{t>0}t^{p}\mu\left(x\in\mathbb{R}^{n}\,:\,\left|f\left(x\right)\right|>t\right)$.

Because of the occurrence of weights and weighted function spaces
let us first define what we mean when we refer to a general weight.
A weight $\omega$ is a function that satisfies the conditions $\omega\in L_{1}^{loc}$
where $\omega>0$ almost everywhere. We denote the weighted Lebesgue
spaces by $L_{p}\left(\mathbb{R}^{n},\omega\right):=\left\{ f:\Vert f\vert L_{p}\left(\mathbb{R}^{n},\omega\right)\Vert<\infty\right\} $,
with their corresponding norm $\Vert f\vert L_{p}\left(\mathbb{R}^{n},\omega\right)\Vert:=\left(\int_{\mathbb{R}^{n}}\vert f\left(x\right)\vert^{p}\omega\left(x\right)dx\right)^{1/p}$
for a weight $\omega$. 

Without further explanation, we will use $p'$ for any given $0<p\leq\infty$,
such that $\frac{1}{p}+\frac{1}{p'}=1$ with the usual agreement,
that if $0<p\leq1,$ then $p'=\infty$ or vice versa. Unless stated
otherwise we will use $L_{p}:=L_{p}\left(\mathbb{R}^{n}\right)$,
$L_{p}\left(\omega\right):=L_{p}\left(\mathbb{R}^{n},\omega\right)$
and also $L_{p}\left(\Omega\right)$ and $L_{p}\left(\Omega,\omega\right)$
whenever the Lebesgue space is restricted to a bounded domain $\Omega\subset\mathbb{R}^{n}$.
We will also take for granted that all function spaces mentioned above
are quasi-Banach spaces, i.e. Banach spaces for $p\geq1$ and quasi-Banach
spaces for $0<p<1$.

Whenever we speak of a ball with radius $R>0$ centered at $x_{0}\in\mathbb{R}^{n}$
we will write $B_{R}\left(x_{0}\right)$ and use the notation $Q_{r}\left(x_{0}\right)$
for paraxial cubes with side length $r>0$ centered at $x_{0}\in\mathbb{R}^{n}$
accordingly. Moreover we also use the notation $\lambda Q$ whenever
the side length of the cube $Q$ is multiplied with $\lambda>0$,
i.e. $\lambda Q_{r}\left(x_{0}\right):=Q_{\lambda r}\left(x_{0}\right)$
(or the radius of the ball $B$ respectively). Also for matters of
abbreviation we write $\left|K\right|:=\mu\left(K\right)$ if $K$
is a measurable subset of $\mathbb{R}^{n}$ with respect to the Lebesgue
measure and by $\chi_{K}$ we refer to the characteristic function,
i.e. $\chi_{K}\left(x\right)=1$ if $x\in K$ and $0$ if $x\notin K$.
For a weight $\omega$ and a measurable set $K$ we use the notation
$\omega\left(K\right):=\int_{K}\omega\left(x\right)dx$. Whenever
we write $a\sim b$, we mean that there are positive constants $c,C>0$
such that $ca\leq b\leq Ca$. Furthermore we will denote the Hardy-Littlewood
maximal operator with $\mathcal{M}$. It is normally defined as 

\[
\mathcal{M}f\left(x\right):=\sup_{R>0}\frac{1}{\left|B_{R}\left(x\right)\right|}\int_{B_{R}\left(x\right)}\left|f\left(y\right)\right|dy.
\]

Finally we remark that throughout the thesis $C$ will denote a positive
constant, which can be different even in a single chain of inequalities
unless explicitly stated otherwise or if a more specific use of constants
and their dependencies is required for the current purpose. 

\newpage{}

In the following we will list some well known definitions and results
that we will use throughout the thesis.
\begin{defn}
We say an operator $T\,:\,L_{p}\left(\mathbb{R}^{n}\right)\rightarrow w-L_{q}\left(\mathbb{R}^{n}\right)$
is of weak-type $\left(p,q\right)$ if

\[
\Vert Tf\vert w-L_{q}\left(\mathbb{R}^{n}\right)\Vert\leq C\cdot\Vert f\vert L_{p}\left(\mathbb{R}^{n}\right)\Vert
\]

holds for all $f\in L_{p}\left(\mathbb{R}^{n}\right)$. 

An operator is said to be of strong-type $\left(p,q\right)$ if

\[
\Vert Tf\vert L_{q}\left(\mathbb{R}^{n}\right)\Vert\leq C\cdot\Vert f\vert L_{p}\left(\mathbb{R}^{n}\right)\Vert
\]

holds for all $f\in L_{p}\left(\mathbb{R}^{n}\right)$.
\end{defn}

\begin{lem}
\label{lem:HoelderUndMinkowski}Let $1<p<\infty$ and $\Omega\subset\mathbb{R}^{n}$
be a (not necessarily bounded) domain. If $f,g\in L_{p}\left(\Omega\right)$
and $h\in L_{p'}\left(\Omega\right)$ then 

\[
\Vert f\cdot h\vert L_{1}\left(\Omega\right)\Vert\leq\Vert f\left(x\right)\vert L_{p}\left(\Omega\right)\Vert\cdot\Vert h\left(x\right)\vert L_{p'}\left(\Omega\right)\Vert\mbox{\,(Hölder's inequality)}
\]
\[
\mbox{and}
\]
\[
\Vert f+g\vert L_{p}\left(\Omega\right)\Vert\leq\Vert f\vert L_{p}\left(\Omega\right)\Vert+\Vert g\vert L_{p}\left(\Omega\right)\Vert\,\mbox{(Minkowski's inequality)}.
\]
\end{lem}

\begin{proof}
One proof can be found in \cite[Sect. 1.3.4.; Lemma]{H. Triebel}.
\end{proof}
\begin{rem}
\label{WeightedHoelder}We should not leave the fact uncommented,
that Hölder's inequality also holds for weighted Lebesgue spaces.
For two functions $f\in L_{r}\left(\Omega\right)$ and $h\in L_{r'}\left(\Omega\right)$,
a weight $\omega$, an open set $\Omega$ and $r>1$ with its dual
$r'$, we have
\begin{eqnarray*}
\int_{\Omega}f\left(x\right)h\left(x\right)\omega\left(x\right)dx & = & \int_{\Omega}f\left(x\right)\omega\left(x\right)^{1/r}h\left(x\right)\omega\left(x\right)^{1/r'}dx\\
 & \leq & \left(\int_{\Omega}f\left(x\right)^{r}\omega\left(x\right)dx\right)^{1/r}\left(\int_{\Omega}h\left(x\right)^{r'}\omega\left(x\right)dx\right)^{1/r'}.
\end{eqnarray*}
Using Hölder's inequality, we can also see that Minkowski's inequality
holds for weighted Lebesgue spaces. Therefore additionally let $g\in L_{r}\left(\Omega\right)$:
\begin{eqnarray*}
 &  & \int_{\Omega}\left|f\left(x\right)+g\left(x\right)\right|^{r}\omega\left(x\right)dx\\
 & \leq & \int_{\Omega}\left(\left|f\left(x\right)\right|+\left|g\left(x\right)\right|\right)\omega\left(x\right)^{1/r}\left|f\left(x\right)+g\left(x\right)\right|^{r-1}\omega\left(x\right)^{1/r'}dx\\
 & \overset{\tiny\mbox{Hölder}}{\leq} & \left(\left(\int_{\Omega}\left|f\left(x\right)\right|^{r}\omega\left(x\right)dx\right)^{1/r}+\left(\int_{\Omega}\left|g\left(x\right)\right|^{r}\omega\left(x\right)dx\right)^{1/r}\right)\\
 &  & \cdot\left(\int_{\Omega}\left|f\left(x\right)+g\left(x\right)\right|^{\left(r-1\right)r'}\omega\left(x\right)dx\right)^{1/r'}\\
 & = & \left(\Vert f\vert L_{r}\left(\omega\right)\Vert+\Vert g\vert L_{r}\left(\omega\right)\Vert\right)\frac{\Vert f+g\vert L_{r}\left(\omega\right)\Vert^{r}}{\Vert f+g\vert L_{r}\left(\omega\right)\Vert}.
\end{eqnarray*}

Now multiplying both sides with $\frac{\Vert f+g\vert L_{r}\left(\omega\right)\Vert^{r}}{\Vert f+g\vert L_{r}\left(\omega\right)\Vert}$
yields the weighted inequality.

\newpage{}
\end{rem}

\begin{thm}
\label{MarcinkiewiczInterpolationTheorem}(Marcinkiewicz Interpolation
Theorem) Let $1\leq p_{0}<p_{1}\leq\infty$ and an operator $T:\mathbb{R}^{n}\rightarrow\mathbb{R}^{n}$
be sublinear and of weak-type $\left(p_{0},p_{0}\right)$ and weak-type
$\left(p_{1},p_{1}\right)$. Then $T$ is of strong-type $\left(p,p\right)$
for all $p_{0}<p<p_{1}$.
\end{thm}

\begin{proof}
The reader may check \cite[Ch. 2; Thm. 2.4.]{J. Duoandikoetxea} for
the proof of an even more general version of this theorem.
\end{proof}
\begin{rem}
To demonstrate an application of an extended version of the Marcinkiewicz
interpolation theorem, we will have a look at the proof of the boundedness
of the Fourier transform 
\[
\mathcal{F}f\left(\xi\right):=\int_{\mathbb{R}^{n}}f\left(x\right)e^{-2\pi ix\cdot\xi}dx.
\]
After expanding the Fourier transform to $L_{2}\left(\mathbb{R}^{n}\right)$,
where we use that the space $L_{1}\left(\mathbb{R}^{n}\right)\cap L_{2}\left(\mathbb{R}^{n}\right)$
is dense in $L_{2}\left(\mathbb{R}^{n}\right)$ one can show that
it is an isometry on $L_{2}\left(\mathbb{R}^{n}\right)$ using Plancharel's
theorem. Since the transform is also a bounded operator from $L_{1}\left(\mathbb{R}^{n}\right)$
to $L_{\infty}\left(\mathbb{R}^{n}\right)$, one can use an extended
version of the Marcinkiewicz interpolation theorem to show that the
Fourier transform is a bounded operator from $L_{p}\left(\mathbb{R}^{n}\right)$
to $L_{p'}\left(\mathbb{R}^{n}\right)$ for $1\leq p\leq2$. For a
more accurate investigation on this topic, the reader may check either
one of the following: \cite[p. 11 et seqq.]{J. Duoandikoetxea}, \cite[Sect. 2.2.]{L. Grafakos }
or \cite[Sect. 1.2.]{J. Bergh; J. L=0000F6fstr=0000F6m}.
\end{rem}

\begin{prop}
\label{FeffermanSteinInequality}(Fefferman Stein Inequality)

Let $f$ be a function and $\varphi\geq0$, then the following inequality
holds for $t>0$:
\[
\int_{\left\{ x:\,\mathcal{M}f\left(x\right)>t\right\} }\varphi\left(x\right)dx\leq\frac{C}{t}\int_{\mathbb{R}^{n}}\left|f\left(x\right)\right|\left(\mathcal{M}\varphi\right)\left(x\right)dx.
\]
\end{prop}

\begin{proof}
The proof can be found in \cite{C. Fefferman; E. M. Stein}. 
\end{proof}

\part{\label{part:Part1}Unweighted Morrey Spaces }

Morrey spaces, named after Charles Bradfield Morrey Jr., can be understood
as generalizations of the Lebesgue spaces in the sense that they serve
the purpose of describing local behavior of the $L_{p}$-norms by
introducing a new parameter $u$. This makes them especially useful
for the specification of solutions of non-linear partial differential
equations. If the parameters $u$ and $p$ are equal, the Morrey space
$M_{u,p}$ coincides with the Lebesgue space $L_{p}$ as we will show
later on. However if the parameter $p$ is smaller than $u$, then
the Lebesgue spaces are continuously embedded in the Morrey space.

Although the topic of this thesis deals with weighted Morrey spaces,
it is only natural to have a look at their unweighted precursors first,
to get an idea of basic properties and relations. At first we give
a definition of (unweighted) Morrey spaces and establish some properties
mainly following \cite{M. Rosenthal} before we continue dealing with
their weighted counterparts.\newpage{}
\begin{defn}
(Morrey spaces) Let $0<p\leq u\leq\infty$. We then call
\[
M_{u,p}:=M_{u,p}\left(\mathbb{R}^{n}\right):=\left\{ f\in L_{p}^{loc}\left(\mathbb{R}^{n}\right):\Vert f\vert M_{u,p}\Vert<\infty\right\} 
\]
Morrey spaces, with their corresponding quasi-norm 

\[
\Vert f\vert M_{u,p}\Vert=\sup_{x\in\mathbb{R}^{n},R>0}R^{n(1/u-1/p)}\left(\int_{B_{R}\left(x\right)}\vert f(y)\vert^{p}dy\right)^{1/p}.
\]
For $p=\infty$ we define the norm $\Vert\cdot\vert M_{\infty,\infty}\Vert$
with respect to the norm of essentially bounded functions analogously.
\end{defn}

Here $\Vert\cdot\vert M_{u,p}\Vert$ is a norm for $p\geq1$ and a
quasi-norm for $0<p<1$. Exactly as it is done in \cite{M. Rosenthal},
we list some basic properties in the following Lemma. It can be found
in \cite[Rem. 1.2.]{M. Rosenthal} where a further reference to \cite[Thm. 4.3.6.]{S. Fucik; O. John; A. Kufner}
is made.
\begin{lem}
\label{lem:(Embedding C)}(i) If $u<p$ then $M_{u,p}=\left\{ 0\right\} $.

(ii) If $u=\infty$ then $M_{\infty,p}=L_{\infty}$ with equivalence
of the corresponding norms.

(iii) If $u=p$ we have $M_{p,p}=L_{p}$.
\end{lem}

\begin{proof}
We refer to \cite[Rem. 1.2.]{M. Rosenthal} for the proof.
\end{proof}
The last fact can be understood in the sense that Morrey spaces are
a generalization of Lebesgue spaces. Since the aim of the thesis is
to expand facts already known for the latter spaces, we shall henceforth
only consider the case $0<p\leq u$ unless stated elsewhere (although
we will mainly focus on $1<p\leq u$). We now proceed with some basic
relations among Morrey Spaces, which are directly taken from \cite[Thm. 1.3.]{M. Rosenthal}.
There the author refers to \cite[Rem. 1.2.;  Lem. 1.7.]{H. Kozono; M. Yamazaki}.
\begin{lem}
\label{lem:LemmaMorreySpaces}(i) $\left(M_{u,p}\vert\Vert\cdot\vert M_{u,p}\Vert\right)$
is a quasi-Banach Space.

(ii) For $0<p_{2}\leq p_{1}\leq u$, we have $M_{u,p_{1}}\hookrightarrow M_{u,p_{2}}$
and furthermore $L_{u}\hookrightarrow M_{u,p_{1}}$.

(iii) For $0<p<u$ we get $w-L_{u}\hookrightarrow M_{u,p}$.

(iv) For $0<p_{1}<u_{1}$ and $0<p_{2}<u_{2}$ let $\frac{1}{p_{1}}+\frac{1}{p_{2}}=:\frac{1}{p}$
and $\frac{1}{u_{1}}+\frac{1}{u_{2}}=:\frac{1}{u}$. If $u\in M_{u_{1},p_{1}}$
and $v\in M_{u_{2},p_{2}}$ we get $uv\in M_{u,p}$ or $\Vert uv\vert M_{u,p}\Vert\leq\Vert u\vert M_{u_{1},p_{1}}\Vert\Vert v\vert M_{u_{2},p_{2}}\Vert$
respectively.
\end{lem}

\begin{proof}
The proofs can be found in \cite[Thm. 1.3.]{M. Rosenthal}.
\end{proof}
In Part \ref{part:WeightedMorreySpaces} we will check if these properties
also hold for weighted Morrey spaces. The main goal of this thesis
is trying to find out if the following result can also be shown if
we add one weight or two respectively. It was inspired by the works
of \cite{L.C. Piccinini} and refines \cite[Thm.2.1.]{L.C. Piccinini},
which deals with Morrey spaces defined on bounded domains. 
\begin{thm}
\label{thm:Sch=0000F6nesTheorem}Let $0<p_{1}\leq u_{1}<\infty$ and
$0<p_{2}\leq u_{2}<\infty$. If $M_{u_{1},p_{1}}\subset M_{u_{2},p_{2}}$,
it follows that $0<p_{2}\leq p_{1}\leq u_{1}=u_{2}$.
\end{thm}

\begin{proof}
The theorem and its proof can be found in \cite[Thm. 1.6.]{M. Rosenthal}.
\end{proof}
The hitherto considered cases did not involve the space $L_{\infty}$,
at which we will have a look at in the following lemma.\newpage{}
\begin{lem}
\label{lem:LemmaUnendlich}(i) If $M_{u,p}\subset L_{\infty}$ then
$u=\infty$.

(ii) If $M_{u,p}\supset L_{\infty}$ then $u=\infty$.
\end{lem}

\begin{proof}
See \cite[Remark 1.7.]{M. Rosenthal} for the proof. 
\end{proof}
This lemma together with Theorem \ref{thm:Sch=0000F6nesTheorem} leads
to the following, very powerful result when it comes to embeddings
of unweighted Morrey spaces.
\begin{cor}
\label{cor:Sch=0000F6nesKorollar}We have $M_{u_{1},p_{1}}\subset M_{u_{2},p_{2}}$
if and only if $0<p_{2}\leq p_{1}\leq u_{1}=u_{2}\leq\infty$.
\end{cor}

\begin{proof}
The corollary is a combination of Theorem \ref{thm:Sch=0000F6nesTheorem}
and Lemma \ref{lem:LemmaUnendlich} and can also be found in \cite[Cor. 1.8.]{M. Rosenthal}.
\end{proof}
Since Morrey spaces are a certain generalization of $L_{p}$ spaces,
it is only natural to wonder how they relate to one another and if
they can be compared. The following theorem gives some insight on
that.
\begin{thm}
\label{thm:unweightedEmbedding}For $0<p<u<\infty$ we have the following
relations
\[
L_{u}\hookrightarrow w-L_{u}\hookrightarrow M_{u,p}\hookrightarrow L_{p}^{loc},
\]
where every embedding is strict.
\end{thm}

\begin{proof}
Again, we refer to \cite[Cor. 1.9.]{M. Rosenthal} for the proof.
\end{proof}
As a last result in this part, we will have a look at a more technically
interesting fact, that might be useful later on. It has been taken
from \cite[Rem. 1.10.]{M. Rosenthal} and goes back to \cite[Prop. 1.3. ;  Lem. 1.6.]{H. Kozono; M. Yamazaki}.
\begin{lem}
(i) If $0<p\leq u$ and $0<r\leq p$ and $f\in M_{u,p}$ then\\
\[
\vert f\vert^{r}\in M_{\frac{u}{r},\frac{p}{r}}\mbox{\,\ and\,\ }\Vert\vert f\vert^{r}\vert M_{\frac{u}{r},\frac{p}{r}}\Vert=\Vert f\vert M_{u,p}\Vert^{r}.
\]
(ii) Suppose we have $0<p\leq u$, $1\leq m<n,$ $y:=\left(y^{1},y^{2}\right)$
where $y^{1}=\left(y_{1,\ldots,}y_{m}\right)$ and $y^{2}=\left(y_{m+1},\ldots y_{n}\right)$.
If $f(y^{1})\in M_{u,p}\left(\mathbb{R}^{m}\right)$ then 
\[
f\in M_{\frac{nq}{m},p}\left(\mathbb{R}^{n}\right),
\]
if we treat $f$ as a function on $\mathbb{R}^{n}$.
\end{lem}

\begin{proof}
Once more, the proof can be found in \cite[Rem. 1.10.]{M. Rosenthal}.
\end{proof}
\newpage{}

\part{\label{part:MuckenhouptWeights}Muckenhoupt Weights}

The next step is to have a closer look at a certain class of weights,
first introduced to sum up all weights $\omega$ for which the Hardy-Littlewood
maximal operator a bounded operator from $L_{p}\left(\omega\right)$
to itself. This class of weights was first introduced by Benjamin
Muckenhoupt in \cite{B. Muckenhoupt}. In Part \ref{part:WeightedMorreySpaces}
we will see that they can also be used to prove embeddings of certain
weighted Morrey spaces. Throughout the whole section we will follow
\cite[Ch. V; §§1,3,5]{E.M. Stein} very closely. 

\section{Definition and Basic Properties}
\begin{defn}
\label{Def: Muckenhoupt}If $1<p<\infty$ is fixed, we say that the
function $\omega:\mathbb{R}^{n}\rightarrow[0,\infty)$ belongs to
the class of\emph{ }Muckenhoupt weights $A_{p}$ if $\omega\in L_{1}^{loc}\left(\mathbb{R}^{n}\right)$,
if it is absolutely continuous $d\mu(x)=\omega(x)dx$ with respect
to the Lebesgue measure $\mu$ and if there is $A>0$ such that $\omega$
satisfies the inequality

\begin{equation}
\left[\frac{1}{\vert B\vert}\int_{B}\omega\left(x\right)dx\right]\cdot\left[\frac{1}{\vert B\vert}\int_{B}\omega\left(x\right)^{-p'/p}dx\right]^{p/p'}\leq A<\infty\label{eq:MuckenHouptDef}
\end{equation}
for all balls $B\subset\mathbb{R}^{n}$. The smallest constant $A$
satisfying this inequality is denoted by $A_{p}(\omega)$, which allows
us to identify all $\omega$ within the class $A_{p}$.

This definition and the following results have been taken from \cite[Ch. V; §1]{E.M. Stein}.
We summarize the bits we need in the next lemma.
\end{defn}

\begin{lem}
\label{thm:BasicApFacts}(i) If $\omega$ belongs to $A_{p}$, then
so do its dilates $\omega_{\delta}\left(x\right)=\omega\left(\delta x\right)$,
its translates $\omega_{h}=\omega\left(\cdot-h\right)$ or its multiplications
with a positive scalar $\lambda\omega\left(\cdot\right)$.

(ii) If $w\in A_{p}$ then $\sigma=\omega^{-p'/p}$ belongs to $A_{p'}$.
Hence $A_{p'}\left(\sigma\right)^{1/p'}=A_{p}\left(\omega\right)^{1/p}$. 

(iii) The classes $A_{p}$ are increasing with $p$, i.e. for $p_{1}<p_{2}$
we get $A_{p_{1}}\subset A_{p_{2}}$.
\end{lem}

\begin{proof}
(i) The result easily follows by substituting the variable and since
inequality \eqref{eq:MuckenHouptDef} must hold for all balls $B\subset\mathbb{R}^{n}$.

(ii) This fact is obvious, if we keep in mind that $1<p<\infty$ and
that $p'$ is the dual of $p$ as well as $\left(p'\right)'=p$. Roughly
speaking we reverse the order of the left and right hand terms in
\eqref{eq:MuckenHouptDef}. We should also denote that in the case
of $p=2$ the result is particularly elegant, i.e.

\[
\left[\frac{1}{\vert B\vert}\int_{B}\omega\left(x\right)dx\right]\cdot\left[\frac{1}{\vert B\vert}\int_{B}\omega\left(x\right)^{-1}dx\right]\leq A.
\]

(iii) Assume that $p_{1}<p_{2}$ and $\omega\in A_{p_{1}}$. It follows
directly as a consequence of Definition \ref{Def: Muckenhoupt} and
Hölder's inequality, which we apply to the right term of the just
mentioned definition with $p=\frac{p_{1}'p_{2}}{p_{1}p_{2}'}$ . We
also need the fact that $q_{2}<q_{1}$ if we define $q_{j}=\frac{p_{j}'}{p_{j}}$.
By that we naturally also get $A_{p_{2}}\left(\omega\right)\leq A_{p_{1}}\left(\omega\right)$.
\end{proof}
\newpage{}
\begin{defn}
(i) A weight $\omega$ belongs to the Muckenhoupt class $A_{1}$ if
there exists a constant $0<A<\infty$ such that 
\[
\frac{1}{\left|B\right|}\int_{B}\omega\left(y\right)dy\leq A\omega\left(x\right)
\]
holds for almost every $x\in\mathbb{R}^{n}$ and all balls $B$. 

(ii) The Muckenhoupt class $A_{\infty}$ is defined by
\[
A_{\infty}=\bigcup_{p\geq1}A_{p}
\]
The following remark and \cite[Def. 2.1.]{D. Haroske; I. Piotrowska}
justify this definition of $A_{1}$.
\end{defn}

\begin{rem}
\label{Remark Definition A_1}Since we know that the spaces $A_{p}$
are increasingly embedded for every $1<p<\infty$, it's only natural
to wonder what is happening to $A_{p}$ if $p\rightarrow1$. We notice
that the second term $\left[\frac{1}{\vert B\vert}\int_{B}\left(\omega\left(x\right)\right)^{-p'/p}dx\right]^{p/p'}\rightarrow\Vert\omega^{-1}\vert L_{\infty}\left(B\right)\Vert$
as $p\rightarrow1$, because it is bounded from above by $\Vert\omega^{-1}\vert L_{\infty}\left(B\right)\Vert$
and the norms $\Vert\omega^{-1}\vert L_{\frac{p'}{p}}\left(B\right)\Vert$
are monotonically increasing in $p$ for fixed balls $B$. However,
this essentially means
\begin{equation}
\frac{1}{\vert B\vert}\int_{B}\omega\left(y\right)dy\leq A\omega\left(x\right)\,\,\mbox{for almost every \,}x\in B.\label{eq:A_1 Absch=0000E4tzung}
\end{equation}
Furthermore we have $\left[\frac{1}{\vert B\vert}\int_{B}\omega\left(x\right)^{-p'/p}dx\right]^{p/p'}\leq\Vert\omega^{-1}\vert L_{\infty}\left(B\right)\Vert$
for all balls $B$. Hence $A_{1}\subset A_{p}$ for $p>1$, so that
the increasing embedding also holds.
\end{rem}

\begin{example}
\label{example1}Let us now have a look at basic examples for Muckenhoupt
weights.

(i) The simplest of all is $\omega^{*}\equiv1$, for which we obviously
see that $\omega^{*}\in A_{p}$ for all $1\leq p\leq\infty$. 

(ii) The next example is simply given by $\omega_{\alpha}\left(x\right)=\left|x\right|^{\alpha}$,
which belongs to a Muckenhoupt class $A_{p}$ if and only if $-n<\alpha<n\left(p-1\right)$,
$1<p<\infty$. The proof is given by the special case $\alpha=\beta$
of the next example.

(iii) The last example we will mention here, is given by
\[
\omega_{\alpha,\beta}\left(x\right)=\begin{cases}
\left|x\right|^{\alpha} & \mbox{, }x<1\\
\left|x\right|^{\beta} & \mbox{, }x\geq1
\end{cases}\mbox{ for }\alpha,\beta\in\mathbb{R},\,x\in\mathbb{R}^{n}.
\]
Then $\omega_{\alpha,\beta}\in A_{p}$ if and only if $-n<\alpha,\beta<n\left(p-1\right)$,
$1<p<\infty$.
\end{example}

\begin{proof}
This elaborate proof of the statement is taken from \cite[1.3.2. Examples]{F. Baaske }. 

``$\Longrightarrow$'' Let $1<p<\infty$ and $\alpha,\beta\in\mathbb{R}$.
We begin by assuming that $\omega_{\alpha,\beta}\in A_{p}$. This
means, that there is a constant $A>0$ such that 

\begin{equation}
\left(\frac{1}{\left|B\right|}\int_{B}\omega_{\alpha,\beta}\left(x\right)dx\right)^{1/p}\left(\frac{1}{\left|B\right|}\int_{B}\omega_{\alpha,\beta}\left(x\right)^{-p'/p}dx\right)^{1/p'}\leq A<\infty,\label{eq:ThisIsPesh}
\end{equation}
for all balls $B$. In particular this includes the ball $B_{R}\left(0\right)$
for $R>0$. First let $0<R<1$. \newpage By also considering the Jacobian
the left hand integral becomes

\[
\int_{B}\omega_{\alpha,\beta}\left(x\right)dx\sim\int_{0}^{R}r^{\alpha+n-1}dr.
\]
We see that for $\alpha=-n$ the integral diverges and since we are
dealing with $R<1$ the same would apply for $\alpha<-n$, because
then $\int_{0}^{R}r^{\alpha+n-1}dr=R^{-\left|\alpha+n\right|}\rightarrow\infty$
as $R\rightarrow0$, leaving us with the condition $\alpha>-n$ as
the only possibly choice for the above integral to be finite. Similarly
we get the right hand integral in \eqref{eq:ThisIsPesh} as 
\[
\int_{B}\omega_{\alpha,\beta}\left(x\right)^{-p'/p}dx\sim\int_{0}^{R}r^{-\alpha\frac{p'}{p}+n-1}dr.
\]
Here the integral obviously diverges for $\alpha=n\left(p-1\right)$.
For $\alpha>n\left(p-1\right)$ the integral becomes $\int_{0}^{R}r^{-\alpha\frac{p'}{p}+n-1}dr=R^{-\left|\alpha\frac{p'}{p}+n\right|}\rightarrow\infty$
as $R\rightarrow0$. This leaves $\alpha<n\left(p-1\right)$ as the
only admissible case in which the integral is finite. If we combine
both conditions, we are left with a certain span for $\alpha$, which
is $-n<\alpha<n\left(p-1\right)$. For this configuration we are left
with
\[
\mbox{\eqref{eq:ThisIsPesh}}\sim R^{-n}R^{\left(\alpha+n\right)/p}R^{\left(-\alpha\frac{p'}{p}+n\right)/p'}=R^{0}<\infty,
\]
thus $R<1$ and $\omega_{\alpha,\beta}\in A_{p}$ both imply $-n<\alpha<n\left(p-1\right)$. 

We proceed by having a look at the case $R>1$, while we still assume
that $\alpha$ satisfies $-n<\alpha<n\left(p-1\right)$. Again, we
will have a look at the ball of radius $R$ centered at the origin
and calculate the integrals to get 
\begin{eqnarray*}
\int_{B}\omega_{\alpha,\beta}\left(x\right)dx & = & \int_{\vert x\vert<1}\omega_{\alpha,\beta}\left(x\right)dx+\int_{1<\vert x\vert<R}\omega_{\alpha,\beta}\left(x\right)dx\\
 & \sim & \int_{0}^{1}r^{\alpha+n-1}dr+\int_{1}^{R}r^{\beta+n-1}dr\\
 & \sim & 1+\begin{cases}
R^{\beta+n}-1 & \mbox{, }\beta>-n\\
1-R^{-\left|\beta+n\right|} & \mbox{, }\beta<-n\\
\log R & \mbox{, }\beta=-n
\end{cases}\\
 & \sim & \begin{cases}
R^{\beta+n} & \beta>-n\\
1 & \beta<-n\\
\log R & \beta=-n
\end{cases},
\end{eqnarray*}
and likewise
\[
\int_{B}\omega_{\alpha,\beta}\left(x\right)^{-p'/p}dx\sim\begin{cases}
R^{-\beta\frac{p'}{p}+n} & \beta<n\left(p-1\right)\\
1 & \beta>n\left(p-1\right)\\
\log R & R=n\left(p-1\right)
\end{cases}.
\]

We can sum up these results as follows: Let $B$ be a ball and define 

\[
I:=\left(\frac{1}{\left|B\right|}\int_{B}\omega_{\alpha,\beta}\left(x\right)dx\right)^{1/p}\left(\frac{1}{\left|B\right|}\int_{B}\omega_{\alpha,\beta}\left(x\right)^{-p'/p}dx\right)^{1/p'}
\]

\newpage then we get the following $5$ cases for different choices
of $\beta$.\\
\\
Case \#1: If $\beta>n\left(p-1\right)>-n$, \\
then $I\sim R^{-n}R^{\frac{\beta+n}{p}}=R^{\frac{1}{p}\left(\beta-n\left(p-1\right)\right)}\rightarrow\infty$
as $R\rightarrow\infty$.\\
\\
Case \#2: If $\beta=n\left(p-1\right)>-n$, \\
then $I\sim R^{-n}R^{\frac{\beta+n}{p}}\left(\log R\right)^{\frac{1}{p'}}=\left(\log R\right)^{\frac{1}{p'}}\rightarrow\infty$
as $R\rightarrow\infty$.\\
\\
Case \#3: If $n\left(p-1\right)>\beta>-n$, \\
then $I\sim R^{-n}R^{\frac{\beta+n}{p}}R^{\left(-\beta\frac{p'}{p}+n\right)\frac{1}{p'}}=R^{0}=1$\\
\\
Case \#4: If $n\left(p-1\right)>\beta=-n$, \\
then $I\sim R^{-n}\left(\log R\right)^{\frac{1}{p}}R^{-\frac{\beta}{p}+\frac{n}{p'}}=\left(\log R\right)^{\frac{1}{p}}\rightarrow\infty$
as $R\rightarrow\infty.$\\
\\
Case \#5: If $n\left(p-1\right)>-n>\beta$, \\
then $I\sim R^{-n}R^{-\frac{\beta}{p}+\frac{n}{p'}}=R^{-\frac{\beta+n}{p}}\rightarrow\infty$
as $R\rightarrow\infty.$\\

However this means that the only admissible $\beta$ are found in
Case 3, concluding one direction of the proof.\\

``$\Longleftarrow$'' Let $-n<\alpha,\beta<n\left(p-1\right)$.
Now we have to find a constant such that the supremum over all balls
$B$ of $I$ is bounded from above. We already saw in the first step
of the proof that balls, centered at $0$ of arbitrary radius $R$
are bounded for this choice of parameters $\alpha$ and $\beta$.
So let $B=B_{R}\left(x_{0}\right)$ for $R>0$ and $x_{0}\in\mathbb{R}^{n}$.
We will divide the proof in three parts, beginning with\\

\emph{Part 1:} $\left|x_{\text{0}}\right|<\frac{R}{2}$.

For this choice the inclusion $B_{\frac{1}{2}R}\left(0\right)\subset B_{R}\left(x_{0}\right)\subset B_{\frac{3}{2}R}\left(0\right)$
holds, making it possible to estimate the integral from both below
and above by 
\begin{eqnarray}
 &  & \frac{c}{\left|B_{\frac{1}{2}R}\left(0\right)\right|}\int_{B_{\frac{1}{2}R}\left(0\right)}\omega_{\alpha,\beta}\left(x\right)dx\nonumber \\
 & < & \frac{1}{\left|B_{R}\left(x_{0}\right)\right|}\int_{B_{R}\left(x_{0}\right)}\omega_{\alpha,\beta}\left(x\right)dx\label{eq:BallIntegralUngleichung}\\
 & < & \frac{C}{\left|B_{\frac{3}{2}R}\left(0\right)\right|}\int_{B_{\frac{3}{2}}\left(0\right)}\omega_{\alpha,\beta}\left(x\right)dx\nonumber 
\end{eqnarray}
Now for $R<\frac{2}{3}$ we have $\frac{1}{2}R<\frac{3}{2}R<1$ and
for $R>2$ we have $\frac{3}{2}R>\frac{1}{2}R>1$ respectively. Thus
the same arguments as for balls centered at $0$ hold, meaning we
know that the supremum over all such balls $B$ of $I$ is bounded
from above (as we have shown in the first step of the proof). What
remains to show is the boundedness for $\frac{2}{3}\leq R\leq2$.
This is done by estimating $\left|B_{R}\left(x_{0}\right)\right|$
in the manner of above to get $c\left(\frac{2}{3}\right)^{n}\leq\left|B_{R}\left(x_{0}\right)\right|\leq C2^{n}$.
We estimate further to see 
\begin{eqnarray}
\frac{1}{\left|B_{R}\left(x_{0}\right)\right|}\int_{B}\omega_{\alpha,\beta}\left(x\right)dx & \geq & C\int_{B_{\frac{R}{2}}\left(0\right)}\omega_{\alpha,\beta}\left(x\right)dx\nonumber \\
 & \geq & C\int_{B_{\frac{1}{3}\left(0\right)}}\omega_{\alpha,\beta}\left(x\right)dx\label{eq:Absch=0000E4tzungObenExample}\\
 & \geq & C\nonumber 
\end{eqnarray}
and 

\begin{eqnarray}
\frac{1}{\left|B_{R}\left(x_{0}\right)\right|}\int_{B}\omega_{\alpha,\beta}\left(x\right)dx & \leq & C\int_{B_{\frac{3}{2}R}\left(0\right)}\omega_{\alpha,\beta}\left(x\right)dx\nonumber \\
 & \leq & C\int_{B_{3\left(0\right)}}\omega_{\alpha,\beta}\left(x\right)dx\label{eq:Absch=0000E4tzungUntenExample}\\
 & \leq & C\nonumber 
\end{eqnarray}
respectively. The whole process can be repeated for $\int_{B}\omega\left(x\right)^{-p'/p}dx$,
which ultimately yields $\sup_{B_{R}\left(x_{0}\right);\left|x_{0}\right|<\frac{R}{2}}I\leq A<\infty$
and brings us to \\

\emph{Part 2: }$\left|x_{0}\right|>2R$.

We begin this part by using the triangle inequality for $x\in B_{R}\left(x_{0}\right)$,
to see
\begin{eqnarray*}
\left|x\right| & \leq & \left|x-x_{0}\right|+\left|x_{0}\right|\leq\frac{3}{2}\left|x_{0}\right|\mbox{ and }\\
\left|x\right| & \geq & \left|x_{0}\right|-\left|x-x_{0}\right|\geq\frac{1}{2}\left|x_{0}\right|.
\end{eqnarray*}
This essentially means that $\omega_{\alpha,\beta}\left(x\right)\sim\omega_{\alpha,\beta}\left(x_{0}\right)$
for all $x\in B_{R}\left(x_{0}\right).$ As above, for $\left|x_{0}\right|>2$
we get $\frac{3}{2}\left|x_{0}\right|>\frac{1}{2}\left|x_{0}\right|>1$
and likewise for $\left|x_{0}\right|<\frac{2}{3}$ we see that the
inequality $\frac{1}{2}\left|x_{0}\right|<\frac{3}{2}\left|x_{0}\right|<1$
holds. So for all $x\in B_{R}\left(x_{0}\right)$ we either have $\left|x\right|<1$
or $\left|x\right|>1$. Thus
\begin{eqnarray*}
I & \sim & \left|x_{0}\right|^{\alpha\left(\frac{1}{p}-\frac{1}{p}\right)}\sim C\mbox{\,\ for }\left|x\right|<1\,\mbox{or}\\
I & \sim & \left|x_{0}\right|^{\beta\left(\frac{1}{p}-\frac{1}{p}\right)}\sim C\,\mbox{ for }\left|x\right|\geq1.
\end{eqnarray*}
Now for $\frac{2}{3}<\left|x_{0}\right|<2$ we see for every $x\in B_{R}\left(x_{0}\right)$
that $1<\left|x\right|<3$. The remaining estimate is 
\[
\min\left(\left(\frac{1}{3}\right)^{\alpha},1,3^{\beta}\right)\leq\omega_{\alpha,\beta}\left(x\right)\leq\max\left(\left(\frac{1}{3}\right)^{\alpha},1,3^{\beta}\right)
\]
for all such $x$. This step and the above yield $\sup_{B_{R}\left(x_{0}\right);\left|x_{0}\right|>2R}I\leq A<\infty$
and we conclude this proof with\newpage{} 

\emph{Part 3: }$\frac{1}{2}R<\left|x_{0}\right|<2R$.

We apply the triangle inequality again to show the boundedness of
the supremum over all balls $B$ of $I$ centered in these $x_{0}$,
to see
\[
\left|x\right|\leq\left|x-x_{0}\right|+\left|x_{0}\right|<3R,
\]
or likewise $B_{R}\left(x_{0}\right)\subset B_{3R}\left(0\right)$.
We can also find an estimate from below. Since $B_{R}\left(x_{0}\right)$
is open, there exists another ball such that $B_{R'}\left(x_{0}\right)\subset B_{R}\left(x_{0}\right)$
with $R'=\frac{R}{6}$. By assumption $\left|x_{0}\right|>\frac{1}{2}R>\frac{1}{3}R>2R'$. 

Now remembering both, the first step of this proof and part 2 we get
\begin{eqnarray}
 &  & \frac{c}{\left|B_{\frac{R}{6}}\left(x_{0}\right)\right|}\int_{B_{\frac{R}{6}}\left(x_{0}\right)}\omega_{\alpha,\beta}\left(x\right)dx\nonumber \\
 & \leq & \frac{1}{\left|B_{R}\left(x_{0}\right)\right|}\int_{B_{R}\left(x_{0}\right)}\omega_{\alpha,\beta}\left(x\right)dx\label{eq:SoaD}\\
 & \leq & \frac{C}{\left|B_{3R}\left(0\right)\right|}\int_{B_{3R}\left(x_{0}\right)}\omega_{\alpha,\beta}\left(x\right)dx.\nonumber 
\end{eqnarray}

We have shown the existence of a boundary for both the above and below
estimates, so that there also exist one for the middle term in \eqref{eq:SoaD}.
Part 1 - 3 finally yield
\[
\sup_{B}\left(\frac{1}{\left|B\right|}\int_{B}\omega_{\alpha,\beta}\left(x\right)dx\right)^{1/p}\left(\frac{1}{\left|B\right|}\int_{B}\omega_{\alpha,\beta}\left(x\right)^{-p'/p}dx\right)^{1/p'}\leq A<\infty.
\]
\end{proof}

\section{Boundedness of the Hardy-Littlewood Maximal Operator on $L_{p}\left(\omega\right)$}

As we have mentioned in the beginning, Muckenhoupt weights were a
result to the question of the boundedness of the Hardy-Littlewood
maximal operator. Our main aim in this section will now be to show,
that for $1<p<\infty$ and $\omega\in A_{p}$ the maximal operator
$\mathcal{M}$ is indeed bounded as follows

\[
\int_{\mathbb{R}^{n}}\left(\mathcal{M}f\left(x\right)\right)^{p}\omega\left(x\right)dx\leq A\int_{\mathbb{R}^{n}}\vert f\left(x\right)\vert^{p}\omega\left(x\right)dx,
\]
for all $f\in L_{p}\left(\omega\right)$. The proof of the boundedness
of the operator from $L_{p}\left(\omega\right)$ to itself contains
an application of the Marcinkiewicz interpolation theorem. However
to apply it according to our use, we first need the following proposition.
\begin{prop}
\label{prop:MaximalOperator of weak-type Lp}Let $d\mu$ be a non-negative
Borel measure and let $1\leq p<\infty$. Then the maximal operator
$\mathcal{M}:f\mapsto\mathcal{M}\left(f\right)$ is of weak-type $\left(p,p\right),$
i.e. 
\[
\mu\left\{ x:\mathcal{M}f\left(x\right)>\alpha\right\} \leq\frac{A}{\alpha^{p}}\int\vert f\left(x\right)\vert^{p}d\mu\left(x\right),\,\mbox{all }\alpha>0,
\]
if and only if $d\mu$ is absolutely continuous, $d\mu\left(x\right)=\omega\left(x\right)dx$
with $\omega\in A_{p}$.
\end{prop}

\begin{proof}
This proposition was taken from \cite[Ch. V; §2.2.]{E.M. Stein},
to which we also refer for the proof. 
\end{proof}
We will also need two further preliminaries to prove the theorem and
the first of them is the reverse Hölder inequality.
\begin{prop}
\label{prop:ReverseH=0000F6lderInequality}(Reverse Hölder Inequality)
If $\omega\in A_{\infty}$ then there exist $r>1$ and $c>0$ such
that 
\begin{equation}
\left[\frac{1}{\vert B\vert}\int_{B}\omega\left(x\right)^{r}dx\right]^{1/r}\leq\frac{c}{\vert B\vert}\int_{B}\omega\left(x\right)dx,\label{eq:ReverseH=0000F6lder}
\end{equation}
for all balls $B$.
\end{prop}

\begin{proof}
We refer again to \cite[Ch. V; §3 Prop.3]{E.M. Stein} for the proof. 
\end{proof}
Taking Proposition \ref{prop:ReverseH=0000F6lderInequality} and Lemma
\ref{thm:BasicApFacts} (iii) we get the following corollary taken
from \cite[Ch. V; Sect. 3 Cor.]{E.M. Stein}.
\begin{cor}
\label{cor:KorollarKleinereAp} Suppose that $\omega\in A_{p}$ for
some $1<p<\infty$ . Then there is $p_{1}<p$ such that $\omega\in A_{p_{1}}$.
\end{cor}

\begin{proof}
With $p'$ being the dual to $p$, we have $p'=\frac{p}{p-1}$. Also
recall Lemma \ref{thm:BasicApFacts} (ii) and define $\sigma=\omega^{-p'/p}\in A_{p'}\subset A_{\infty}$.
By the above Proposition \ref{prop:ReverseH=0000F6lderInequality}
we know that $\sigma$ satisfies a reverse Hölder inequality for some
$r>1$
\begin{equation}
\left[\frac{1}{\vert B\vert}\int_{B}\sigma\left(x\right)^{r}dx\right]^{1/r}\leq\frac{c}{\vert B\vert}\int_{B}\sigma\left(x\right)dx.\label{eq: KorollarKleinereAp}
\end{equation}
Now we see that $r\cdot p'/p=r/\left(p-1\right)=1/\left(p_{1}-1\right)=p_{1}'/p_{1}$
for some $1<p_{1}<p$, because $r>1.$ Furthermore $\omega\in A_{p}$
by assumption and by rearranging the inequality in Definition  \ref{Def: Muckenhoupt},
taking the $p/p'$-th power of \eqref{eq: KorollarKleinereAp} and
multiplying both sides with $\omega\left(B\right)/\left|B\right|$
yields $\omega\in A_{p_{1}}$. 
\end{proof}
The last two proofs were preparations for the earlier stated main
aim of this section. We will now prove it in the next theorem, taken
from \cite[Ch. V; Sect. 3; Thm. 1]{E.M. Stein}.
\begin{thm}
\label{thm:MainTheoremAp}Suppose $1<p<\infty$ and $\omega\in A_{p}$.
Then\\
\begin{equation}
\int_{\mathbb{R}^{n}}\left(\mathcal{M}f(x)\right)^{p}\omega\left(x\right)dx\leq A\int_{\mathbb{R}^{n}}\vert f\left(x\right)\vert^{p}\omega\left(x\right)dx,\label{eq:MaximalInequality}
\end{equation}

for all $f\in L_{p}\left(\omega\right)$.
\end{thm}

\begin{proof}
Let $\omega\in A_{p}$ be given. Then by Corollary \ref{cor:KorollarKleinereAp}
there is a $1<p_{1}<p$ for which $\omega\in A_{p_{1}}$. We know
by Proposition \ref{prop:MaximalOperator of weak-type Lp} that $\mathcal{M}$
is of weak-type $\left(L_{p_{1}}\left(\omega\right),L_{p_{1}}\left(\omega\right)\right)$.
It is also obvious that $\mathcal{M}$ is bounded on $L_{\infty}$$\left(\omega\right)$.
Now together with the Marcinkiewicz interpolation theorem (Theorem
\ref{MarcinkiewiczInterpolationTheorem}) it follows that $\mathcal{M}$
is bounded on $L_{p}\left(\omega\right)$ to itself (since $1<p_{1}<p<\infty$). 
\end{proof}
By using the Marcinkiewicz interpolation theorem, we make use of two
interpolation points (i.e. $p_{1}$ and $\infty$). In Proposition
\ref{prop:MaximalOperator of weak-type Lp} we have shown that the
weak-type $\left(p_{1},p_{1}\right)$ inequality only holds for $\omega\in A_{p_{1}}$.
Hence Theorem \ref{thm:MainTheoremAp} gives us the maximal possible
class of weights one can find when we use this interpolation theorem. 
\begin{rem}
\label{BoundednessOfMomega}Another problem which can be solved is
the question of whether or not $\mathcal{M}$ is bounded from $L_{q}\left(\omega\right)$
to $L_{q}\left(\nu\right)$. However, we won't investigate this particular
result in this thesis. It turns out that $\mathcal{M}$ is bounded
from $L_{q}\left(\omega\right)$ to $L_{q}\left(\nu\right)$ for every
$q$ with $1<p<q<\infty$ if $\nu,\omega\in A_{p}$. We refer to \cite[Ch. IV; Cor. 1.13]{J. Garcia-Cuerva; J.L. Rubio de Francia}
for the proof.
\end{rem}

\section{More Results for $A_{p}$}

There are, of course, more facts which can be derived for weights
belonging to a Muckenhoupt class $A_{p}$ but only a few will be mentioned
in this thesis. A detail, which we will need later on is the fact
that every Muckenhoupt weight is also a doubling measure. We will
prove this in the following.
\begin{defn}
\label{Def: DoublingMeasure}(Doubling Measure) We say a weight $\omega$
is a doubling measure if there is a constant $C>0$ such that
\[
\omega\left(B_{2R}\left(x_{0}\right)\right)\leq C\omega\left(B_{R}\left(x_{0}\right)\right)
\]
for all $x_{0}\in\mathbb{R}^{n}$ and $R>0$.
\end{defn}

The following lemma is an alternative way of defining the class of
Muckenhoupt weights. In this way of defining it, the origin as the
class of weights for which the maximal operator is bounded is more
obvious. It is taken from \cite[Ch. V; §1; 1.4]{E.M. Stein}.
\begin{lem}
\label{lem:AlternativeDarstellungFuerAp} Let $f\in L_{1}^{loc}$
and $B$ be a ball, then $\omega\in A_{p}$ for $p>1$  if and only
if 
\begin{equation}
\left(\frac{1}{\left|B\right|}\int_{B}f\left(x\right)dx\right)^{p}\leq\frac{c}{\omega\left(B\right)}\int_{B}f\left(x\right)^{p}\omega\left(x\right)dx\label{eq:LemmaResult}
\end{equation}
holds for all non-negative functions $f$ and all balls $B$.
\end{lem}

\begin{proof}
Let $B$ be a ball and define $f_{B}:=\frac{1}{\left|B\right|}\int_{B}f\left(x\right)dx$
as the mean value of $f$ on $B$.

``$\Longrightarrow$'' Assume $\omega\in A_{p}$, then we get
\[
f_{B}=\frac{1}{\left|B\right|}\int_{B}f\left(x\right)dx=\frac{1}{\left|B\right|}\int f\left(x\right)\omega\left(x\right)^{1/p}\omega\left(x\right)^{-1/p}dx,
\]
and apply Hölder's inequality with the exponents $p$ and $p'$ to
get
\[
\left(f_{B}\right)^{p}\leq\left|B\right|^{-p}\left(\int_{B}f\left(x\right)^{p}\omega\left(x\right)dx\right)\left(\int_{B}\omega\left(x\right)^{-p'/p}\right)^{p/p'}.
\]
From there on we see that \eqref{eq:MuckenHouptDef} leads to the
desired \eqref{eq:LemmaResult}.

``$\Longleftarrow$'' Conversely, assume \eqref{eq:LemmaResult}
holds. Our first intention is to set $f=\omega^{-p'/p}$, because
then $f^{p}\omega=\omega^{-p'+1}=\omega^{-p'/p}$ and rearranging
\eqref{eq:LemmaResult} would yield exactly the wanted inequality
\eqref{eq:MuckenHouptDef}. However, we do not know whether or not
$\int_{B}\omega\left(x\right)^{-p'/p}dx$ is finite and hence we replace
$f$ by $\left(\omega+\varepsilon\right)^{-p'/p}$ with $\varepsilon>0$.
Now by assumption $\left(\frac{1}{\left|B\right|}\int_{B}\left(\omega+\varepsilon\right)\left(x\right)^{-p'/p}dx\right)^{p}\leq\frac{C}{\omega\left(B\right)}\int_{B}\left(\omega+\varepsilon\right)\left(x\right)^{-p'}\omega\left(x\right)dx$
. Next we let $\varepsilon\rightarrow0$, use the Lebesgue monotone
convergence theorem (see \cite[App. 1, Thm. 2]{H. Triebel}), and
rearrange the inequality to get
\[
\left[\frac{1}{\left|B\right|}\int_{B}\omega\left(x\right)dx\right]\left[\frac{1}{\left|B\right|}\int_{B}\omega\left(x\right)^{-p'}\omega\left(x\right)dx\right]^{p/p'}\leq C<\infty
\]
(see \cite[Sect. 9.1.1.]{L. Grafakos }). With $\omega^{-p'+1}=\omega^{-p'/p}$
this yields \eqref{eq:MuckenHouptDef}.
\end{proof}
The following small insertion uses content of \cite[p. 133 et sqq.]{J. Duoandikoetxea}.
\begin{rem}
\label{Remark: AboutDuo}Lemma \ref{lem:AlternativeDarstellungFuerAp}
tells us two interesting facts about Muckenhoupt weights. Given a
measurable subset $D$ of a ball $B$, i.e. $D\subset B$ , we take
\eqref{eq:LemmaResult} with $f\equiv\chi_{D}$, which yields
\[
\omega\left(B\right)\left(\frac{\left|D\right|}{\left|B\right|}\right)^{p}\leq\omega\left(D\right).
\]
From this inequality we deduce the two facts. Firstly $\omega\equiv0$
or $\omega>0$ a.e., for if a measurable subset $D$ (which we assume
to be bounded) of a ball $B$ has the property $\omega\vert_{D}=0$
then it follows that $\omega\left(B\right)=0$. Secondly, $\omega\in L_{1}^{loc}$
or $\omega=\infty$ a.e.. In case there is a ball $B$ for which $\omega\left(B\right)=\infty$,
the same holds for every ball containing $B$ as well as any measurable
set contained in $B$.
\end{rem}

\begin{cor}
\label{cor:ApDoublingMeasure}Let $\omega\in A_{p}$, then $\omega$
also is a doubling measure. 
\end{cor}

\begin{proof}
Let $\omega\in A_{p}$. Now let $R>0$ and set $B:=B_{2R}\left(x_{0}\right)$
as well as $f:=\chi_{B_{R}\left(x_{0}\right)}$ to apply Lemma \ref{lem:AlternativeDarstellungFuerAp}.
Clearly \eqref{eq:LemmaResult} now turns into
\[
\left(\frac{1}{\left|B_{2R}\left(x_{0}\right)\right|}\int_{B_{2R}\left(x_{0}\right)}\chi_{B_{R}\left(x_{0}\right)}\left(x\right)dx\right)^{p}\leq\frac{c}{\omega\left(B_{2R}\left(x_{0}\right)\right)}\int_{B_{2R}\left(x_{0}\right)}\left(\chi_{B_{R}}\left(x\right)\right)^{p}\omega\left(x\right)dx.
\]
From there on we easily deduce
\[
\omega\left(B_{2R}\left(x_{0}\right)\right)\leq c'\omega\left(B_{R}\left(x_{0}\right)\right),
\]
with $c'=2^{np}c$. Hence $\omega$ is a doubling measure.
\end{proof}
The next theorem shows how to construct Muckenhoupt weights $\omega$
belonging to the class $A_{p}$ for every $1<p<\infty$, given two
weights $\omega_{1},\omega_{2}\in A_{1}$ or vice versa. The theorem
as well as its proof have been taken from \cite[Ch. V; Sect. 5.3. Prop. 9]{E.M. Stein}.
\begin{thm}
Suppose that $\omega_{1},\omega_{2}\in A_{1}$. If $1\leq p<\infty$,
then $\omega=\omega_{1}\omega_{2}^{1-p}$ belongs to $A_{p}$. Equivalently,
suppose $\omega\in A_{p}$ , then there are $\omega_{1},\omega_{2}\in A_{1}$
such that $\omega=\omega_{1}\omega_{2}^{1-p}$.
\end{thm}

\begin{proof}
\emph{Step 1 (Motivation)}

To get a first idea of how this proof works, let us consider the most
transparent case, which is $p=2$. Let $\omega_{1},\omega_{2}\in A_{1}$.
We begin by having a look at $\omega=\omega_{1}\omega_{2}^{-1}$ and
further 
\[
\int_{B}\omega\left(x\right)dx\leq\int_{B}\omega_{1}\left(x\right)dx\left[\inf_{x\in B}\omega_{2}\left(x\right)\right]^{-1}.
\]
Likewise
\[
\int_{B}\omega^{-1}\left(x\right)dx\leq\int_{B}\omega_{2}\left(x\right)dx\left[\inf_{x\in B}\omega_{1}\left(x\right)\right]^{-1},
\]
from which we conclude
\[
\left(\frac{1}{\vert B\vert}\int_{B}\omega\left(x\right)dx\right)\left(\frac{1}{\vert B\vert}\int_{B}\omega\left(x\right)^{-1}dx\right)\leq\prod_{i=1}^{2}\left(\frac{1}{\vert B\vert}\int_{B}\omega_{i}\left(x\right)dx\left[\inf_{x\in B}\omega_{i}\left(x\right)^{-1}\right]\right),
\]
where each term on the right hand side is bounded, since $\omega_{i}\in A_{1}$
for $i=1,2$. This yields $\omega\in A_{2}$.

Let us now prove the converse. We assume that $\omega\in A_{2}$ and
take the maximal operator $\mathcal{M}$. Furthermore we construct
\begin{equation}
Tf:=\omega^{-1/2}\mathcal{M}\left(\omega^{1/2}f\right)+\omega^{1/2}\mathcal{M}\left(\omega^{-1/2}f\right).\label{eq:OperatorT}
\end{equation}
By definition of $\mathcal{M}$ this operator is sublinear, i.e. $T\left(f+g\right)\leq Tf+Tg$.
It is also symmetric in $\omega$, i.e. if $\omega\in A_{2}$ then
so is $\omega^{-1}$. So we can assume that $\omega^{-1/2}\mathcal{M}\omega^{1/2}$
and $\omega^{1/2}\mathcal{M}\omega^{-1/2}$ are bounded from $L_{2}\left(\mathbb{R}^{n}\right)\rightarrow L_{2}\left(\mathbb{R}^{n}\right)$.
Together with \eqref{eq:OperatorT} we conclude
\begin{equation}
\Vert Tf\vert L_{2}\left(\mathbb{R}^{n}\right)\Vert\leq A\Vert f\vert L_{2}\left(\mathbb{R}^{n}\right)\Vert\,\mbox{ for some\,}A>0.\label{eq:Operatorabsch=0000E4tzung}
\end{equation}
Let us now fix $\Vert f\vert L_{2}\left(\mathbb{R}^{n}\right)\Vert=1$
and define 
\[
\eta:=\sum_{k=1}^{\infty}\left(2A\right)^{-k}T^{k}\left(f\right),
\]
where we used the notation for the iterated operator $T^{k}f:=T^{k-1}\left(T\left(f\right)\right)$.
From there on we easily see that $\Vert\eta\vert L_{2}\left(\mathbb{R}^{n}\right)\Vert\leq\sum_{k=1}^{\infty}2^{-k}=1$,
since it shrinks down to a geometric series together with \eqref{eq:Operatorabsch=0000E4tzung}.
Hence $\eta\in L_{2}\left(\mathbb{R}^{n}\right)$. The next fact we
need to take into consideration is the pointwise inequality
\[
T\eta=\sum_{k=1}^{\infty}\left(2A\right)^{-k}T^{k+1}\left(f\right)\leq\sum_{k=1}^{\infty}\left(2A\right)^{1-k}T^{k}\left(f\right)=\left(2A\right)\eta.
\]
Now with $\omega_{1}=\omega^{1/2}\eta$, we see
\[
\mathcal{M}\left(\omega_{1}\right)\leq T\left(\eta\right)\omega^{1/2}\leq\left(2A\right)\eta\omega^{1/2}=\left(2A\right)\omega_{1},
\]
and hence essentially $\omega_{1}\in A_{1}$ as we have seen in \eqref{eq:A_1 Absch=0000E4tzung}.
Likewise we look at $\omega_{2}=\omega^{-1/2}\eta$ and by applying
the same steps we get $\mathcal{M}\left(\omega_{2}\right)\leq A\omega_{2}$
and $\omega_{2}\in A_{1}$. We conclude 
\[
\frac{\omega_{1}}{\omega_{2}}=\frac{\omega^{1/2}\eta}{\omega^{-1/2}\eta}=\omega\in A_{2},
\]
and thus the case $p=2$ is proved.

Now that we gained some insight for the arguments needed to prove
the claim, let us proceed to general $p$.

\emph{Step 2 $"\Longrightarrow"$}

The proof of the first part is completely similar to what we did earlier
for $p=2$. So let $\omega_{1},\omega_{2}\in A_{1}$ and $\omega=\omega_{1}\omega_{2}^{1-p}$.
As before we notice 
\[
\int_{B}\omega\left(x\right)dx\leq\int\omega_{1}\left(x\right)dx\left[\inf_{x\in B}\omega_{2}\left(x\right)\right]^{1-p}\mbox{ and}
\]
\[
\int_{B}\omega\left(x\right)^{-p'/p}dx\leq\int_{B}\omega_{2}\left(x\right)dx\left[\inf_{x\in B}\omega_{1}\left(x\right)\right]^{-p'/p}
\]
and hence
\begin{eqnarray*}
 & \left[\frac{1}{\left|B\right|}\int_{B}\omega\left(x\right)dx\right]\left[\frac{1}{\left|B\right|}\int_{B}\omega\left(x\right)^{-p'/p}dx\right]^{p/p'}\\
\leq & \frac{1}{\left|B\right|}\int_{B}\omega_{1}\left(x\right)dx\left[\inf_{x\in B}\omega_{2}\left(x\right)\right]^{1-p}\left[\frac{1}{\left|B\right|}\int_{B}\omega_{2}\left(x\right)dx\right]^{p/p'}\left[\inf_{x\in B}\omega_{1}\left(x\right)\right]^{-1} & .
\end{eqnarray*}

\emph{Step 3 $"\Longleftarrow"$ }

The converse is also done equivalently to the case $p=2$. Let us
start with $p\geq2$ and $\omega\in A_{p}$. We define
\[
Tf=\left[\omega^{-1/p}\mathcal{M}\left(f^{p/p'}\omega^{1/p}\right)\right]^{p'/p}+\omega^{1/p}\mathcal{M}\left(f\omega^{-1/p}\right).
\]
Because $\omega\in A_{p}$ we know $\omega^{-p'/p}\in A_{p'}$ by
Lemma \ref{thm:BasicApFacts}. To show that $T$ is bounded on $L_{p}$
we need to recall Theorem \ref{thm:MainTheoremAp}. Minkowski's inequality
(see Lemma \ref{lem:HoelderUndMinkowski}) gives $T\left(f+g\right)\leq Tf+Tg,$
since $p\geq2$ and hence $p/p'=p/\left(p/p-1\right)\geq1$. The fact
that both $\left[\omega^{-1/p}M\left(f^{p/p'}\omega^{1/p}\right)\right]^{p'/p}$
and $\omega^{1/p}M\left(f\omega^{-1/p}\right)$ are bounded from $L_{p}\left(\mathbb{R}^{n}\right)$
to itself is again obvious if we recall Theorem \ref{thm:MainTheoremAp}.
With all these preparations, we can now form $\eta$ just like above
and write $\omega_{1}=\omega^{1/p}\eta^{p/p'}$ and $\omega_{2}=\omega^{-1/p}\eta$.
Also just like above $T\eta\leq\left(2A\right)\eta$ for some $A>0$
and proceeding $\omega_{1},\omega_{2}\in A_{1}$. Furthermore
\[
\omega_{1}\omega_{2}^{1-p}=\omega^{1/p}\eta^{p/p'}\cdot\left(\omega^{-1/p}\eta\right)^{1-p}=\omega,
\]
which concludes the proof for $p\geq2$. To finish the proof for $p<2$,
we simply use that $\omega^{-p'/p}\in A_{p'}$, factorize it just
like above and raise the product to the power $\left(-p/p'\right)$,
so that the factorization holds for $\omega$.
\end{proof}

\begin{rem}
Although it will not occur again in this thesis, we should mention
the existence of more general Muckenhoupt classes used for other purposes.
They occurred first in \cite{B. Muckenhoupt; R. L. Wheeden} as a
result to the question of boundedness for all non-negative functions
$\omega$ for the operator $T_{\gamma}f\left(x\right):=\int f\left(x-y\right)\left|y\right|^{\gamma-n}dy$,
with $n$ being the dimension. In particular the problem was to determine
all non-negative functions $\omega$ for which $\Vert T_{\gamma}f\vert L_{q}\left(\omega\right)\Vert\leq C\Vert f\vert L_{p}\left(\omega\right)\Vert$
holds. The result was a class of weight functions denoted with $A_{p,q}$.
A precise formulation of the problem, as well as a proper definition
can be found in \cite{B. Muckenhoupt; R. L. Wheeden}. The definition
and a few facts connecting $A_{p,q}$ to the regular $A_{p}$ can
also be found in \cite[Def. 2.10.]{Y. Komori; S. Shirai}. 

A weight function belongs to $A_{p,q}$ for $1<p<q<\infty$ if there
exists $C>1$ such that
\[
\left(\frac{1}{\vert B\vert}\int_{B}\omega\left(x\right)^{q}dx\right)^{1/q}\left(\frac{1}{\vert B\vert}\int_{B}\omega\left(x\right)^{-p'}dx\right)^{1/p'}\leq C.
\]
For $p=1$ we say that $\omega\in A_{1,q}$ with $1<q<\infty$ if
there is $C>1$ such that
\[
\left(\frac{1}{\vert B\vert}\int_{B}\omega\left(x\right)^{q}dx\right)^{1/q}\left(\underset{x\in B}{\mbox{ess }\sup}\frac{1}{\omega\left(x\right)}\right)\leq C.
\]
\end{rem}

\newpage{}

\part{\label{part:WeightedMorreySpaces}Weighted Morrey Spaces}

After we studied unweighted Morrey spaces and Muckenhoupt weights
we will now have a look at weighted Morrey spaces. Whereas we will
focus on general weights in Section \ref{sec:Section4}, we will mostly
combine Part \ref{part:Part1} and Part \ref{part:MuckenhouptWeights}
in Section \ref{sec:Section5} and \ref{sec:Embeddings-of-weighted}
and have a look at weighted Morrey spaces using Muckenhoupt weights. 

\section{\label{sec:Section4}Definition and Basic Properties}

We begin by defining weighted Morrey spaces and compare them to the
definition and the results in Part \ref{part:Part1}.
\begin{defn}
\label{Def:WeightedMorreySpaces}(Weighted Morrey Spaces) Let $0<p\leq u\leq\infty$,
and $\omega$ be a weight. Then the weighted Morrey spaces are defined
by 
\[
M_{u,p}\left(\omega\right)=\left\{ f\in L_{p}^{loc}\left(\omega\right):\Vert f\vert M_{u,p}\left(\omega\right)\Vert<\infty\right\} ,
\]
with their corresponding quasi-norm
\[
\Vert f\vert M_{u,p}\left(\omega\right)\Vert=\sup_{x\in\mathbb{R}^{n},R>0}\omega\left(B_{R}\left(x\right)\right)^{\frac{1}{u}-\frac{1}{p}}\left(\int_{B_{R}\left(x\right)}\vert f\left(y\right)\vert^{p}\omega\left(y\right)dy\right)^{1/p}.
\]
For $p=\infty$ we define the norm $\Vert\cdot\vert M_{\infty,\infty}\left(\omega\right)\Vert$
with respect to the norm of essentially bounded functions analogously.
\end{defn}

As usual $\Vert\cdot\vert M_{u,p}\left(\omega\right)\Vert$ is a norm
for $p\geq1$ and a quasi-norm for $p<1$, as we will show in the
next lemma.. When we have a look at the definition of weighted Morrey
spaces and recall the facts for their unweighted counterparts from
Part \ref{part:Part1}, we are able to immediately transfer certain
features without too much work. For the analogue proofs, we refer
to Lemma \ref{lem:(Embedding C)} and \ref{lem:LemmaMorreySpaces}.
\begin{lem}
\label{Re: RepeatedRemark}(i) If $u<p$, then $M_{u,p}\left(\omega\right)=\left\{ 0\right\} $.

(ii) If $u=\infty$, then $M_{\infty,p}\left(\omega\right)=L_{\infty}\left(\omega\right)$
with equivalence of the corresponding norms.

(iii) If $u=p$ we have $M_{p,p}\left(\omega\right)=L_{p}$$\left(\omega\right)$.

(iv) $\left(M_{u,p}\left(\omega\right)\vert\Vert\cdot\vert M_{u,p}\left(\omega\right)\Vert\right)$
is a quasi-Banach space.
\end{lem}

\begin{proof}
We will only prove (iv). As we have already mentioned in Lemma \ref{lem:(Embedding C)},
the proofs of (i) and (ii) can be found in \cite[Rem. 1.2.]{M. Rosenthal}
for the unweighted case, but can be adapted when we add a weight.
(iii) follows immediately by the definition of the norm of weighted
Morrey spaces. The proof of (iv) follows \cite[Thm. 1.3.]{M. Rosenthal}
(i) almost exactly. By Lemma \ref{Re: RepeatedRemark} (ii) it is
enough to focus on $p<\infty$.

At first we prove that $\Vert\cdot\vert M_{u,p}\left(\omega\right)\Vert$
is a norm for $p\geq1$ and a quasi-norm for $0<p<1$. It is obvious
that $\Vert f\vert M_{u,p}\left(\omega\right)\Vert\geq0$ for all
$f\in M_{u,p}\left(\omega\right)$. Now since $\omega>0$ a.e. we
have \\
\begin{eqnarray*}
\Vert f\vert M_{u,p}\left(\omega\right)\Vert=0 & \Longleftrightarrow & \Vert f\vert L_{p}\left(B,\omega\right)\Vert=0\mbox{ for all balls }B\\
 & \Longleftrightarrow & \Vert f\vert L_{p}\left(\omega\right)\Vert=0\\
 & \Longleftrightarrow & f=0\mbox{ a.e.}.
\end{eqnarray*}
The next step $\Vert\lambda f\vert M_{u,p}\left(\omega\right)\Vert=\left|\lambda\right|\Vert f\vert M_{u,p}\left(\omega\right)\Vert$
for $\lambda\in\mathbb{R}$ is obvious. At last we see that the triangle
inequality holds for $p\geq1$ by using Minkowski's inequality for
weighted Lebesgue spaces , i.e. 
\begin{eqnarray*}
\Vert f+g\vert M_{u,p}\left(\omega\right)\Vert & = & \sup_{B}\omega\left(B\right)^{1/u-1/p}\Vert f+g\vert L_{p}\left(B,\omega\right)\Vert\\
 & \overset{\tiny\mbox{Lem. }\ref{lem:HoelderUndMinkowski}}{\leq} & \sup_{B}\omega\left(B\right)^{1/u-1/p}\left(\Vert f\vert L_{p}\left(B,\omega\right)\Vert+\Vert g\vert L_{p}\left(B,\omega\right)\Vert\right)\\
 & \leq & \Vert f\vert M_{u,p}\left(\omega\right)\Vert+\Vert g\vert M_{u,p}\left(\omega\right)\Vert.
\end{eqnarray*}
For $p<1$ we use $\Vert f+g\vert L_{p}\left(\omega\right)\Vert\leq2^{1/p-1}\left(\Vert f\vert L_{p}\left(\omega\right)\Vert+\Vert g\vert L_{p}\left(\omega\right)\Vert\right)$
for $f,g\in L_{p}\left(\omega\right)$, which is obtained by adding
a weight to the proof in \cite[Thm. 1.10.]{J. Elstrodt}. Hence we
get $\Vert f+g\vert M_{u,p}\left(\omega\right)\Vert\leq2^{1/p-1}\left(\Vert f\vert M_{u,p}\left(\omega\right)\Vert+\Vert g\vert M_{u,p}\left(\omega\right)\Vert\right)$.

The next fact we need is, that $M_{u,p}\left(\omega\right)$ is a
vector space, which is easily deduced by the fact that $M_{u,p}\left(\omega\right)\subset L_{p}\left(\omega\right)$,
i.e. $M_{u,p}\left(\omega\right)$ is a quasi-normed subspace of the
weighted Lebesgue spaces, thus making it a vector space.

The last thing to show is that $M_{u,p}\left(\omega\right)$ is complete.
We will have a look at $p\geq1$ first. Therefore let $\left(f_{n}\right)_{n\in\mathbb{N}}$
be a Cauchy sequence in $M_{u,p}\left(\omega\right)$. We have to
show the existence of some $f\in M_{u,p}\left(\omega\right)$ with
$\Vert f-f_{n}\vert M_{u,p}\left(\omega\right)\Vert\overset{n\rightarrow\infty}{\longrightarrow}0$.
Now since $\left(f_{n}\right)_{n\in\mathbb{N}}$ is a Cauchy sequence,
we know that there is a subsequence $\left(f_{n_{k}}\right)_{k\in\mathbb{N}}$
such that we have $\Vert f_{n_{k+1}}-f_{n_{k}}\vert M_{u,p}\left(\omega\right)\Vert\leq\frac{1}{2^{k}}$
for $k=1,2,\ldots$. We set $g:=\sum_{k=1}^{\infty}\vert f_{n_{k+1}}-f_{n_{k}}\vert$,
which obviously is non-negative and measurable and show that $g\in M_{u,p}\left(\omega\right)$
as follows:
\begin{eqnarray}
\Vert g\vert M_{u,p}\left(\omega\right)\Vert & = & \sup_{B}\omega\left(B\right)^{1/u-1/p}\left(\int_{B}\left|g\left(x\right)\right|^{p}\omega\left(x\right)dx\right)^{1/p}\nonumber \\
 & = & \sup_{B}\omega\left(B\right)^{1/u-1/p}\left(\int_{B}\left(\sum_{k=1}^{\infty}\left|f_{n_{k+1}}\left(x\right)-f_{n_{k}}\left(x\right)\right|\right)^{p}\omega\left(x\right)dx\right)^{1/p}\nonumber \\
 & = & \sup_{B}\lim_{m\rightarrow\infty}\omega\left(B\right)^{1/u-1/p}\left(\int_{B}\left(\sum_{k=1}^{m}\left|f_{n_{k+1}}\left(x\right)-f_{n_{k}}\left(x\right)\right|\right)^{p}\omega\left(x\right)dx\right)^{1/p}\label{eq:gleqinfty}\\
 & = & \sup_{B}\lim_{m\in\mathbb{N}}\ldots=\lim_{m\in\mathbb{N}}\sup_{B}\ldots\nonumber \\
 & = & \lim_{m\rightarrow\infty}\Vert\sum_{k=1}^{m}\left|f_{n_{k+1}}\left(x\right)-f_{n_{k}}\left(x\right)\right|\vert M_{u,p}\left(\omega\right)\Vert\nonumber \\
 & \leq & \lim_{m\rightarrow\infty}\sum_{k=1}^{m}\Vert f_{n_{k+1}}\left(x\right)-f_{n_{k}}\left(x\right)\vert M_{u,p}\left(\omega\right)\Vert\label{eq:ieq1}\\
 & \leq & 1.\nonumber 
\end{eqnarray}
Since $g$ is non-negative and measurable we can use the Lebesgue
monotone convergence theorem (\cite[App. 1, Thm. 2]{H. Triebel}),
which here allows us to change positions of limes and supremum. As
next step, we show that $g<\infty$ almost everywhere. In \eqref{eq:gleqinfty}
we see that $\int_{B}\left|g\left(x\right)\right|^{p}\omega\left(x\right)dx=\int_{\mathbb{R}^{n}}\left|g\left(x\right)\right|^{p}\chi_{B}\left(x\right)\omega\left(x\right)dx<\infty$
for all balls $B$. We can deduce that $g\chi_{B}$ is finite with
respect to the Lebesgue measure almost everywhere or that there is
a set with measure zero, i.e. $\left|N_{R,x_{0}}\right|=0$, such
that $g\left(x\right)\chi_{B}\left(x\right)<\infty$ for all $x\notin N_{R,x_{0}}$
if we denote $B=B\left(R,x_{0}\right)$. Now we define
\[
N:=\bigcup_{R\in Q_{+},x_{0}\in\mathbb{Q}^{n}}N_{R,x_{0}},
\]
and see $0\leq\mu\left(N\right)\leq\sum_{R\in Q_{+},x_{0}\in\mathbb{Q}^{n}}\mu\left(N_{r,x_{0}}\right)=0$,
which makes $N$ itself a zero set with respect to the Lebesgue measure.
This however yields that $g<\infty$ almost everywhere, because for
every $y\in\mathbb{R}^{n}\backslash N$ there are $R\in\mathbb{Q}_{+}$
and $x_{0}\in\mathbb{Q}^{n}$ such that $y\in B_{R}\left(x_{0}\right)$.
By choice of $y$ and $N$ we know that $y\notin N_{R,x_{0}}$ and
thus $g\left(y\right)<\infty$. With that in mind, we define a set
$H:=\left\{ x\in\mathbb{R}^{n}:\,0\leq g<\infty\right\} $. By our
above considerations $H$ is measurable and $\left|\mathbb{R}^{n}\backslash H\right|=0$.
Since we have shown $g\in M_{u,p}\left(\omega\right)$, we also know
that $g^{*}:=g\chi_{H}\in M_{u,p}\left(\omega\right)$. Now we define
$f$ as follows
\[
f:=\begin{cases}
f_{n_{1}}+\sum_{k=1}^{\infty}\left|f_{n_{k+1}}-f_{n_{k}}\right| & x\in H\\
0 & x\notin H
\end{cases}.
\]
Clearly $f\leq\left|f_{n_{1}}\right|+g^{*}$ and thus $f\in M_{u,p}\left(\omega\right)$.
The last thing we need for the proof of the completeness of $M_{u,p}\left(\omega\right)$,
is that $f_{n}\rightarrow f$ in $M_{u,p}\left(\omega\right)$. Because
$f_{n}$ is a Cauchy sequence, it is sufficient to show that $f_{n_{k}}\rightarrow f$
in $M_{u,p}\left(\omega\right)$. During this last estimation we will
again use the Lebesgue monotone convergence theorem.
\begin{eqnarray}
 &  & \Vert f-f_{n_{m}}\vert M_{u,p}\left(\omega\right)\Vert\nonumber \\
 & = & \sup_{B}\omega\left(B\right)^{1/u-1/p}\left(\int_{B}\left|\sum_{k=m}^{\infty}f_{n_{k+1}}\left(x\right)-f_{n_{k}}\left(x\right)\right|^{p}\omega\left(x\right)dx\right)^{1/p}\nonumber \\
 & \leq & \sup_{B}\omega\left(B\right)^{1/u-1/p}\lim_{l\rightarrow\infty}\left(\int_{B}\left|\sum_{k=m}^{l}f_{n_{k+1}}\left(x\right)-f_{n_{k}}\left(x\right)\right|^{p}\omega\left(x\right)dx\right)^{1/p}\nonumber \\
 & = & \lim_{l\rightarrow\infty}\Vert\sum_{k=m}^{l}f_{n_{k+1}}-f_{n_{k}}\vert M_{u,p}\left(\omega\right)\Vert\nonumber \\
 & \leq & \lim_{l\rightarrow\infty}\sum_{k=m}^{l}\Vert f_{n_{k+1}}-f_{n_{k}}\vert M_{u,p}\left(\omega\right)\Vert\leq\sum_{k=m}^{\infty}\frac{1}{2^{k}}\overset{\tiny m\rightarrow\infty}{\longrightarrow}0.\label{eq:ieq2}
\end{eqnarray}
\end{proof}
To prove the completeness of $M_{u,p}\left(\omega\right)$ for $p<1$
we modify the proof for $p\geq1$ by replacing every use of the triangle
inequality (i.e. \eqref{eq:ieq1} and \eqref{eq:ieq2}) by using a
corresponding $\varrho$-inequality instead (i.e. using $\Vert f+g\Vert^{\varrho}\leq\Vert f\Vert^{\varrho}+\Vert g\Vert^{\varrho}$
for some $\varrho\in\left(0,1\right]$. Here we apply the fact that
for every quasi-norm with constant $C\geq1$ there exists a corresponding
$\varrho$- norm with $\varrho\in\left(0,1\right]$. See \cite[Ch. 2, Thm. 1.1.]{R. DeVore; G.G. Lorentz}
for the proof). The remaining facts of Lemma \ref{lem:LemmaMorreySpaces}
(mostly concerning embeddings) will be answered in Section \ref{sec:Embeddings-of-weighted}. 

Another fact worth transferring to the weighted case is the following
lemma. The proof of the unweighted case was taken from \cite[Lem. 1.4.]{M. Rosenthal},
where a further reference to \cite[Lem. 1.4.]{H. Kozono; M. Yamazaki}
is made. 

\newpage{}
\begin{lem}
\label{lem:PreparationLemmaForuvinMup}Let $m\in\mathbb{N}$, $\omega$
be a weight and $\left(u_{j}\right)_{j=1}^{m}$, $\left(p_{j}\right)_{j=1}^{m}$
with $0<p_{j}\leq u_{j}<\infty$ for $j=1,\ldots,m$ . Furthermore
let $\sum_{j=1}^{m}\frac{1}{p_{j}}\leq\frac{1}{p}$ and $f_{j}\in M_{u_{j},p_{j}}\left(\omega\right)$
for $j=1,\ldots,m$. We define $\frac{1}{u}:=\sum_{j=1}^{m}\frac{1}{u_{j}}$.
Then for $f_{0}\in L_{\infty}$ we have
\[
f_{0}\cdot f_{1}\cdot\ldots\cdot f_{m}\in M_{u,p}\left(\omega\right).
\]
\end{lem}

\begin{proof}
At first, we have $p\leq u$, since 
\[
\frac{1}{u}=\sum_{j=1}^{m}\frac{1}{u_{j}}\leq\sum_{j=1}^{m}\frac{1}{p_{j}}\leq\frac{1}{p}.
\]
With this choice of parameters $p$ and $u$ the space $M_{u,p}\left(\omega\right)$
is non-trivial (also recall Lemma \ref{Re: RepeatedRemark} (i)).
Now let $B$ be a ball. We have to show both $f_{0}\cdot f_{1}\cdot\ldots\cdot f_{m}\in M_{u,p}\left(\omega\right)$
and $f_{0}\cdot f_{1}\cdot\ldots\cdot f_{m}\in L_{p}^{loc}\left(\omega\right)$,
though by showing the first, the latter is implied. We begin by showing
that applying Hölder's inequality $n$-times yields the following:
\begin{equation}
\Vert f_{0}\cdot\ldots\cdot f_{n}\vert L_{p}\left(B\right)\Vert\leq\Vert f_{0}\vert L_{\infty}\left(B\right)\Vert\left(\prod_{i=1}^{m}\Vert f_{i}\vert L_{p_{i}}\left(B\right)\Vert\right)\Vert1\vert L_{1}\left(B\right)\Vert^{1/p-\sum_{j=1}^{m}\frac{1}{p_{j}}}.\label{eq:inequalityforLp1andsoonandsoforth}
\end{equation}
We do so, by induction. Clearly for $m=1$ we have 
\[
\Vert f_{0}\cdot f_{1}\vert L_{p}\left(B\right)\Vert\leq\Vert f_{0}\vert L_{\infty}\left(B\right)\Vert\Vert f_{1}\vert L_{p_{1}}\left(B\right)\Vert\left(\int_{B}1dx\right)^{1/p\cdot1/r'},
\]
where $r'$ is the dual to $r=\frac{p_{1}}{p}$ and hence $r'=\frac{p_{1}}{p_{1}-p}$.
The exponent then is
\[
\frac{1}{p}\cdot\frac{1}{r'}=\frac{1}{p}\cdot\frac{p_{1}-p}{p_{1}}=\frac{1}{p}-\frac{1}{p_{1}},
\]
and as a result $\left(\int_{B}1dx\right)^{1/p\cdot1/r'}=\left|B\right|^{1/p-1/p_{1}}$.
Let us now assume that \eqref{eq:inequalityforLp1andsoonandsoforth}
is true for a fixed $m\in\mathbb{N}$. We  show, that it also holds
for $m+1$:
\begin{eqnarray*}
 &  & \Vert f_{0}\cdot\ldots\cdot f_{m+1}\vert L_{p}\left(B\right)\Vert\\
 & \leq & \Vert f_{0}\vert L_{\infty}\left(B\right)\Vert\prod_{j=1}^{m}\Vert f_{j}\vert L_{p_{j}}\left(B\right)\Vert\left(\int_{B}\vert f_{m+1}\vert^{\left(\frac{1}{p}-\sum_{j=1}^{m}\frac{1}{p_{j}}\right)^{-1}}dx\right)^{\frac{1}{p}-\sum_{j=1}^{m}\frac{1}{p_{j}}}.
\end{eqnarray*}
By that, we get the parameter $r$ for the Hölder inequality as $r=p_{m+1}\left(\frac{1}{p}-\sum_{j=1}^{m}\frac{1}{p_{j}}\right)$.
Its dual then is $r'=\frac{p_{m+1}\left(\frac{1}{p}-\sum_{j=1}^{m}\frac{1}{p_{j}}\right)}{p_{m+1}\left(\frac{1}{p}-\sum_{j=1}^{m}\frac{1}{p_{j}}\right)-1}$.
By having a look at the outside exponent of the integral, we see that
\[
\frac{p_{m+1}\left(\frac{1}{p}-\sum_{j=1}^{m}\frac{1}{p_{j}}\right)-1}{p_{m+1}\left(\frac{1}{p}-\sum_{j=1}^{m}\frac{1}{p_{j}}\right)}\cdot\left(\frac{1}{p}-\sum_{j=1}^{m}\frac{1}{p_{j}}\right)=\frac{1}{p}-\sum_{j=1}^{m+1}\frac{1}{p_{j}},
\]
which concludes the induction. 

The Hölder inequality also holds for weighted $L_{p}$ spaces (see
Remark \ref{WeightedHoelder}). Hence, we can also use the upper induction
when we integrate the weighted volume of the ball, i.e. replacing
$\int_{B}1dx$ with $\int_{B}\omega\left(x\right)dx$. Doing so yields
the following chain of inequalities
\begin{eqnarray*}
 &  & \Vert f_{0}\cdot\ldots\cdot f_{m}\vert M_{u,p}\left(\omega\right)\Vert\\
 & = & \sup_{B}\omega\left(B\right)^{1/u-1/p}\cdot\left(\int_{B}\vert f_{0}\cdot\ldots\cdot f_{m}\vert^{p}\omega\left(x\right)dx\right)^{1/p}\\
 & \leq & \Vert f_{0}\vert L_{\infty}\Vert\sup_{B}c_{p}\cdot\omega\left(B\right)^{\left(\sum_{j=1}^{n}1/u_{j}\right)-1/p}\cdot\prod_{i=1}^{m}\Vert f_{i}\vert L_{p_{i}}\left(B,\omega\right)\Vert\omega\left(B\right)^{1/p-\sum_{j=1}^{m}1/p_{j}}\\
 & = & C_{p}\Vert f\vert L_{\infty}\Vert\sup_{B}\omega\left(B\right)^{\left(\sum_{j=1}^{m}1/u_{j}-1/p_{j}\right)}\cdot\prod_{i=1}^{m}\Vert f_{i}\vert L_{p_{i}}\left(B,\omega\right)\Vert\\
 & \leq & C\Vert f_{0}\vert L_{\infty}\Vert\prod_{i=1}^{m}\Vert f_{i}\vert M_{u_{i},p_{i}}\left(\omega\right)\Vert,
\end{eqnarray*}
which concludes the first step of the proof.

The next step is to show that $f_{0}\cdot f_{1}\cdot\ldots\cdot f_{m}\in L_{p}^{loc}\left(\omega\right)$.
Let $K$ be any compact set, then we can choose $R^{*}>0$ and $x^{*}\in\mathbb{R}^{n}$,
such that $K\subset B_{R^{*}}\left(x^{*}\right)$. Furthermore we
have
\begin{eqnarray*}
 &  & \omega\left(B_{R^{*}}\left(x^{*}\right)\right)^{1/u-1/p}\left(\int_{B_{R^{*}}\left(x^{*}\right)}\left|f_{0}\left(x\right)\cdot\ldots\cdot f_{m}\left(x\right)\right|^{p}\omega\left(x\right)dx\right)^{1/p}\\
 & \leq & \sup_{B}\omega\left(B\right)^{1/u-1/p}\left(\int_{B_{R}}\left|f_{0}\left(x\right)\cdot\ldots\cdot f_{m}\left(x\right)\right|^{p}\omega\left(x\right)dx\right)^{1/p}\\
 & \leq & \Vert f_{0}\cdot\ldots\cdot f_{m}\vert M_{u,p}\left(\omega\right)\Vert\leq C\Vert f_{0}\vert L_{\infty}\Vert\Vert f_{1}\vert M_{u_{1},p_{1}}\left(\omega\right)\Vert\cdot\ldots\cdot\Vert f_{m}\vert M_{u_{m},p_{m}}\left(\omega\right)\Vert,
\end{eqnarray*}
which we have just shown in the previous step. This chain of inequalities
however, shows that 
\[
\int_{K}\left|f_{0}\left(x\right)\cdot\ldots\cdot f_{m}\left(x\right)\right|^{p}\omega\left(x\right)dx\leq\int_{B_{R^{*}\left(x^{*}\right)}}\left|f_{0}\left(x\right)\cdot\ldots\cdot f_{m}\left(x\right)\right|^{p}\omega\left(x\right)dx<\infty,
\]
which concludes the second step and the proof.
\end{proof}
To prevent misunderstandings in Lemma \ref{lem:PreparationLemmaForuvinMup},
we quote from \cite[p. 137]{J. Duoandikoetxea} in the following remark.
\begin{rem}
Lemma \ref{lem:PreparationLemmaForuvinMup} sets $f_{0}\in L_{\infty}$
as a condition, despite the fact that $f_{1},\ldots,f_{n}$ belong
to the weighted spaces $L_{p_{i}}\left(\omega\right)$ for $i=1,\ldots,m$.
Surprisingly the space of essentially bounded functions is equal to
its weighted counterpart, i.e. $L_{\infty}=L_{\infty}\left(\omega\right)$
(including equivalence of norms) for Muckenhoupt weight. To see this,
we only to recall Remark \ref{Remark: AboutDuo}, where we showed
that $\omega\left(Q\right)=0$ if and only if $\left|Q\right|=0$.\newpage{}
\end{rem}

\begin{cor}
\label{cor:DatJanzeHochR}Let $0<r\leq p\leq u$, $\omega$ be a weight
and $f\in M_{u,p}$$\left(\omega\right)$, then 
\[
\left|f\right|^{r}\in M_{\frac{u}{r},\frac{p}{r}}\left(\omega\right)\,\mbox{ and }\,\Vert\left|f\right|^{r}|M_{\frac{u}{r},\frac{p}{r}}\left(\omega\right)\Vert=\Vert f|M_{u,p}\left(\omega\right)\Vert^{r}.
\]
\end{cor}

\begin{proof}
The equivalence of the norms $\Vert\left|f\right|^{r}|M_{\frac{u}{r},\frac{p}{r}}\left(\omega\right)\Vert=\Vert f|M_{u,p}\left(\omega\right)\Vert^{r}$
immediately follows by Definition \ref{Def:WeightedMorreySpaces}.
Also $\Vert\left|f\right|^{r}|M_{\frac{u}{r},\frac{p}{r}}\left(\omega\right)\Vert<\infty$
is deduced from that, so that the only thing left to prove is $\left|f\right|^{r}\in L_{\frac{p}{r}}^{loc}\left(\omega\right)$,
which we will show like in Lemma \ref{lem:PreparationLemmaForuvinMup}.
Let $K$ be a compact set, for which we choose $R^{*}>0$ and $x^{*}\in\mathbb{R}^{n}$
such that $K\subset B_{R^{*}}\left(x^{*}\right)$. Thus
\begin{eqnarray*}
 &  & \omega\left(B_{R^{*}\left(x^{*}\right)}\right)^{\frac{r}{u}-\frac{r}{p}}\left(\int_{B_{R^{*}}\left(x^{*}\right)}\left(\left|f\left(x\right)\right|^{r}\right)^{\frac{p}{r}}\omega\left(x\right)dx\right)^{\frac{r}{p}}\\
 & \leq & \sup_{B}\omega\left(B\right)^{\frac{r}{u}-\frac{r}{p}}\left(\int_{B}\left(\left|f\left(x\right)\right|^{r}\right)^{\frac{p}{r}}\omega\left(x\right)dx\right)^{\frac{r}{p}}\\
 & = & \Vert f\vert M_{u,p}\left(\omega\right)\Vert^{r},
\end{eqnarray*}
from where we know that 
\[
\left(\int_{K}\left(\left|f\left(x\right)\right|^{r}\right)^{\frac{p}{r}}\omega\left(x\right)dx\right)^{\frac{r}{p}}\leq\left(\int_{B_{R^{*}}\left(x^{*}\right)}\left(\left|f\left(x\right)\right|^{r}\right)^{\frac{p}{r}}\omega\left(x\right)dx\right)<\infty
\]
and thus $\left|f\right|^{r}\in L_{\frac{p}{r}}^{loc}\left(\omega\right)$.
Hence $\left|f\right|^{r}\in M_{\frac{u}{r},\frac{p}{r}}\left(\omega\right)$.
\end{proof}
\begin{rem}
Some of the papers we quote and use for the following text refer to
weighted Morrey spaces when using cubes $Q=Q_{r}\left(x_{0}\right)$
instead of balls $B=B_{R}\left(x_{0}\right)$ (e.g. \cite{Y. Komori; S. Shirai}).
In some articles and papers, (weighted) Morrey spaces are introduced
with the quasi-norm
\[
\Vert f\vert M_{u,p}\left(\omega\right)\Vert^{*}=\sup_{x\in\mathbb{R}^{n},r>0}\omega\left(Q_{r}\left(x\right)\right)^{1/u-1/p}\left(\int_{Q_{r}\left(x\right)}\left|f\left(y\right)\right|^{p}\omega\left(y\right)dy\right)^{1/p}
\]
(again $\Vert\cdot\vert M_{u,p}\left(\omega\right)\Vert^{*}$ is a
norm for $p\geq1$ and a quasi-norm for $0<p<1$). We will quickly
show that this quasi-norm is equivalent to the previously defined
one, i.e. $\Vert f\vert M_{u,p}\left(\omega\right)\Vert\sim\Vert f\vert M_{u,p}\left(\omega\right)\Vert^{*}$.
Recall that $\int_{Q_{r}\left(x\right)}1dx=r^{n}$ and $\int_{B_{R}\left(x\right)}1dx=\frac{\pi^{n/2}}{\Gamma\left(\frac{n}{2}+1\right)}R^{n},$
where $\Gamma$ is the gamma function. Obviously we will always find
$c,C>0$ such that $cR^{n}\leq r^{n}\leq CR^{n}$, since the dimension
$n$ is fixed. Roughly speaking, this means that we can always find
two $n$-dimensional balls, such that a fixed $n$-dimensional cube
contains a smaller ball and is contained in a bigger ball or vice
versa. From here it is easy to see that this implies $\omega\left(Q_{r}\left(x\right)\right)\sim\omega\left(B_{R}\left(x\right)\right)$
and $\int_{Q_{r}\left(x\right)}\left|f\left(y\right)\right|^{p}\omega\left(y\right)dy\sim\int_{B_{R}\left(x\right)}\left|f\left(y\right)\right|^{p}\omega\left(y\right)dy$
respectively, since both $\omega\geq0$ and $\left|f\right|\geq0$
a.e.. Henceforth we shall use both norms (i.e. $\Vert\cdot\vert M_{u,p}\left(\omega\right)\Vert$
and $\Vert\cdot\vert M_{u,p}\left(\omega\right)\Vert^{*}$) equivalently,
depending on which one suits the current purpose better.\newpage{}
\end{rem}

\section{\label{sec:Section5}Boundedness of the Hardy-Littlewood Maximal
Operator on $M_{u,p}\left(\omega\right)$}

In analogy to Theorem \ref{thm:MainTheoremAp} we will now have a
look at the Hardy-Littlewood maximal operator (and the weighted maximal
operator) and show its boundedness as a mapping from a weighted Morrey
space to itself. 
\begin{defn}
\label{Def: weightedMaximalOperator} Let $\omega$ be a weight. Then
we define the weighted maximal operator as
\[
\mathcal{M}_{\omega}f\left(x\right)=\sup_{Q\ni x}\frac{1}{\omega\left(Q\right)}\int_{Q}\vert f\left(y\right)\vert\omega\left(y\right)dy.
\]
Now, as we have seen before in Corollary \ref{cor:ApDoublingMeasure},
we know that every Muckenhoupt weight $\omega\in A_{p}$ is also a
doubling measure. There is one more property of doubling measures
we will need in the main proof of this section. It is taken from \cite[Lem. 4.1.]{Y. Komori; S. Shirai}.
\end{defn}

\begin{lem}
\label{lem:LemmaDoubling}Let $\omega$ be a doubling measure, then
there is a constant $D>1$ such that
\[
D\omega\left(Q\right)\leq\omega\left(2Q\right).
\]
\end{lem}

\begin{proof}
We begin by fixing a cube $Q:=Q_{r}\left(x_{0}\right)$. Then we can
choose a cube $K\subset2Q$ with side length $\frac{r}{2}$, which
is disjoint from $Q$. We get
\[
\omega\left(Q\right)+\omega\left(K\right)\leq\omega\left(2Q\right).
\]
For that same $K$, we get $Q\subset5K$, which yields $\omega\left(Q\right)\leq\omega\left(5K\right)\leq C^{3}\omega\left(K\right)$,
where the doubling constant is $C>0$ (see Definition \ref{Def: DoublingMeasure}).
Inserting this in the above estimation yields
\begin{eqnarray*}
\left(1+\frac{1}{C^{3}}\right)\omega\left(Q\right) & = & \omega\left(Q\right)+\frac{\omega\left(Q\right)}{C^{3}}\\
 & \leq & \omega\left(Q\right)+\omega\left(K\right)\\
 & \leq & \omega\left(2Q\right),
\end{eqnarray*}
where $D:=\left(1+\frac{1}{C^{3}}\right)>1$ is the desired constant.

We will also need another lemma to prove the main theorem of this
chapter, . This fact is taken from \cite[Lem. 4.2. (2)]{Y. Komori; S. Shirai}
with further reference to \cite[Sect. 9.1.1.]{L. Grafakos }.
\end{proof}
\begin{lem}
\label{lem:MunweightMweight} Let $\omega\in A_{p}$ for some $1<p<\infty$.
Then
\[
\mathcal{M}f\left(x\right)\leq C\mathcal{M}_{\omega}\left(\vert f\vert^{p}\right)\left(x\right)^{1/p},\,x\in\mathbb{R}^{n}.
\]
\end{lem}

\begin{proof}
At first we fix a ball $B$ and let $f\in L_{p}\left(\omega\right)$.
We know that \eqref{eq:MaximalInequality} holds for $f$. Also we
clearly have
\begin{equation}
\frac{1}{\left|B\right|}\int_{B}\left|f\left(x\right)\right|dx\leq\mathcal{M}\left(f\chi_{B}\right)\left(y\right)\label{eq:Inequality?}
\end{equation}
for all $y\in B$. We weight-integrate the above to the $p$-th power
over all $y\in B$ and apply Theorem \ref{thm:MainTheoremAp} to get
\begin{eqnarray*}
\int_{B}\left(\frac{1}{\left|B\right|}\int_{B}\left|f\left(x\right)\right|dx\right)^{p}\omega\left(y\right)dy & = & \omega\left(B\right)\left(\frac{1}{\left|B\right|}\int_{B}\left|f\left(x\right)\right|dx\right)^{p}\\
 & \overset{\tiny\eqref{eq:Inequality?}}{\leq} & \int_{B}\mathcal{M}\left(f\chi_{B}\right)^{p}\left(y\right)dy\\
 & \overset{\tiny\mbox{Thm. }\ref{thm:BasicApFacts}}{\leq} & C\int_{B}\left|f\left(y\right)\right|^{p}\omega\left(y\right)dy.
\end{eqnarray*}
The lemma follows after dividing by $\omega\left(B\right)$, taking
the $1/p$-th power of the inequality and afterwards taking the supremum
over all such balls $B\ni y$.
\end{proof}
Another result that we need in particular is stated in the following
proposition. It is similar to Theorem \ref{thm:MainTheoremAp} and
shows that the weighted maximal operator is bounded from $L_{p}\left(\omega\right)$
to itself, whenever $\omega$ is a doubling measure. However we won't
prove it but rather refer to \cite[Ch. II; Thm. 2.6.]{J. Garcia-Cuerva; J.L. Rubio de Francia}
for further references.
\begin{prop}
\label{prop:BoundednessOfMweight}Let $\omega$ be a regular positive
Borel measure in $\mathbb{R}^{n}$ which also is a doubling measure.
Then for every $1<p<\infty$, there is a constant $C_{p}>0$ such
that for every $f\in L_{p}\left(\omega\right)$
\[
\int_{\mathbb{R}^{n}}\left(\mathcal{M}_{\omega}f\left(x\right)\right)^{p}\omega\left(x\right)dx\leq C_{p}\left(\int_{\mathbb{R}^{n}}\left|f\left(x\right)\right|^{p}\omega\left(x\right)dx\right).
\]
\end{prop}

\begin{proof}
See \cite[Ch. II; Thm. 2.6.]{J. Garcia-Cuerva; J.L. Rubio de Francia}
for the proof. 
\end{proof}
In the following theorems we will have a look at the boundedness of
the maximal operator from $M_{u,p}\left(\omega\right)$ to itself.
We follow \cite[Sect. 3]{Y. Komori; S. Shirai} closely (though differ
in notation) with the first theorem we quote being \cite[Thm. 3.1.]{Y. Komori; S. Shirai}. 
\begin{thm}
\label{thm:MaximalDoubling}If $1<p<u<\infty$, and $\omega$ a doubling-measure,
then the operator $\mathcal{M}_{\omega}$ is bounded on $M_{u,p}\left(\omega\right)$
to itself. 
\end{thm}

\begin{proof}
First we fix a cube $Q\subset\mathbb{R}^{n}$ and decompose $f=f_{1}+f_{2}$.
We choose $f_{1}:=f\chi_{3Q}$ and $f_{2}$ accordingly. It is easy
to check that $\mathcal{M}_{\omega}$ is a sublinear operator, since
\begin{eqnarray*}
\mathcal{M}_{\omega}\left(f+g\right)\left(y\right) & = & \sup_{Q\ni y}\frac{1}{\omega\left(Q\right)}\int_{Q}\left|f+g\right|\left(x\right)\omega\left(x\right)dx\\
 & \leq & \sup_{Q\ni y}\frac{1}{\omega\left(Q\right)}\int_{Q}\left(\left|f\right|+\left|g\right|\right)\left(x\right)\omega\left(x\right)dx\\
 & \leq & \sup_{Q\ni y}\frac{1}{\omega\left(Q\right)}\int_{Q}\left|f\right|\left(x\right)\omega\left(x\right)dx\\
 &  & +\sup_{Q\ni y}\frac{1}{\omega\left(Q\right)}\int_{Q}\left|g\right|\left(x\right)\omega\left(x\right)dx\\
 & = & \mathcal{M}_{\omega}\left(f\right)\left(y\right)+\mathcal{M}_{\omega}\left(g\right)\left(y\right).
\end{eqnarray*}
With that in mind and Minkowski's inequality (Lemma \ref{lem:HoelderUndMinkowski})
this yields
\begin{eqnarray*}
 &  & \left(\int_{Q}\mathcal{M}_{\omega}f\left(x\right)^{p}\omega\left(x\right)dx\right)^{1/p}\\
 & \leq & \underset{=:T_{1}}{\underbrace{\left(\int_{Q}\mathcal{M}_{\omega}f_{1}\left(x\right)^{p}\omega\left(x\right)dx\right)^{1/p}}}+\underset{=:T_{2}}{\underbrace{\left(\int_{Q}\mathcal{M}_{\omega}f_{2}\left(x\right)^{p}\omega\left(x\right)dx\right)^{1/p}}}.
\end{eqnarray*}
We will now have a look at $T_{1}$ and $T_{2}$ separately and begin
with the former. It is known that $\mathcal{M}_{\omega}$ is a bounded
operator on $L_{p}\left(\omega\right)$ (see Proposition \ref{prop:BoundednessOfMweight})
and since $p>1$
\begin{eqnarray}
T_{1} & \leq & C\left(\int_{\mathbb{R}^{n}}\mathcal{M}_{\omega}f_{1}\left(x\right)^{p}\omega\left(x\right)dx\right)^{1/p}\leq C\left(\int_{3Q}\left|f\left(x\right)\right|{}^{p}\omega\left(x\right)dx\right)^{1/p}\nonumber \\
 & \leq & C\Vert f\vert M_{u,p}\left(\omega\right)\Vert\omega\left(Q\right)^{1/p-1/u},\label{eq:T_1}
\end{eqnarray}
where the last step follows from the doubling measure properties of
$\omega$.

Let us now have a look at $T_{2}$ by considering some geometric properties.
For any $x\in Q$
\[
\mathcal{M}_{\omega}f_{2}\left(x\right)\leq\sup_{R:\,Q\subset3R}\frac{1}{\omega\left(R\right)}\int_{R}\left|f\left(y\right)\right|\omega\left(y\right)dy.
\]
Also for bounded domains, we have
\begin{eqnarray*}
 &  & \frac{1}{\omega\left(R\right)}\int_{R}\left|f\left(y\right)\right|\omega\left(y\right)dy\\
 & \leq & \left(\omega\left(R\right)^{\frac{p}{u}-1}\int_{R}\left|f\left(y\right)\right|{}^{p}\omega\left(y\right)dy\right)^{1/p}\left(\omega\left(R\right)^{1-\frac{p}{u}}\omega\left(R\right)^{-1}\right)^{1/p}\\
 & \leq & C\Vert f\vert M_{u,p}\left(\mathbb{R}^{n},\omega\right)\Vert\omega\left(Q\right)^{-1/u}
\end{eqnarray*}
whenever $Q\subset3R$. Again, we used that $\omega$ is a doubling
measure. This finally yields
\begin{eqnarray*}
T_{2} & \leq & C\cdot\Vert f\vert M_{u,p}\left(\mathbb{R}^{n},\omega\right)\Vert\omega\left(Q\right)^{-1/u}\left(\int_{Q}\omega\left(y\right)dy\right)^{1/p}\\
 & = & C\cdot\Vert f\vert M_{u,p}\left(\mathbb{R}^{n},\omega\right)\Vert\omega\left(Q\right)^{1/p-1/u}.
\end{eqnarray*}
This and \eqref{eq:T_1} ultimately complete the proof, after multiplying
both sides of the inequality with $\omega\left(Q\right)^{1/u-1/p}$
and taking the supremum over all $Q$.
\end{proof}
As we have mentioned before, the next theorem is taken from \cite[Thm. 3.2.]{Y. Komori; S. Shirai}.
\begin{thm}
If $1<p\leq u<\infty$, and $\omega\in A_{p}$, then the Hardy-Littlewood
maximal operator $\mathcal{M}$ is bounded on $M_{u,p}\left(\omega\right)$
to itself. If $1=p<u$ and $\omega\in A_{1}$, then $\mathcal{M}$
is of weak-type $\left(1,1\right)$.\newpage{}
\end{thm}

\begin{proof}
First let $1<p<u<\infty$. For easier notation during this proof,
we will write $\kappa:=1-\frac{p}{u}$. We see that $0<\kappa<1$.
Since $\omega\in A_{p}$ we know because of the reverse Hölder inequality
(Corollary \ref{cor:KorollarKleinereAp}) that there is $r$ such
that $1<r<p$ and $\omega\in A_{r}$ and together with Lemma \ref{lem:MunweightMweight}
and Theorem \ref{thm:MaximalDoubling} we get 
\begin{eqnarray*}
 &  & \left(\omega\left(Q\right)^{-\kappa}\int_{Q}\mathcal{M}f\left(x\right)^{p}\omega\left(x\right)dx\right)^{1/p}\\
 & \overset{\tiny\mbox{Lem.\,}\ref{lem:MunweightMweight}}{\leq} & C\cdot\left(\omega\left(Q\right)^{-\kappa}\int_{Q}\mathcal{M}_{\omega}\left(\vert f\vert^{r}\right)\left(x\right)^{p/r}\omega\left(x\right)dx\right)^{1/p}\\
 & \leq & C\cdot\Vert\mathcal{M}_{\omega}\left(\left|f\right|{}^{r}\right)|M_{u/r,p/r}\left(\omega\right)\Vert^{1/r}\\
 & \overset{\tiny\mbox{Thm. }\ref{thm:MaximalDoubling}}{\leq} & C\cdot\Vert\left|f\right|^{r}\vert M_{u/r,p/r}\left(\omega\right)\Vert^{1/r}\\
 & \overset{\tiny\mbox{Cor. }\ref{cor:DatJanzeHochR}}{\leq} & C\cdot\Vert f\vert M_{u,p}\left(\omega\right)\Vert.
\end{eqnarray*}

Now let $p=1$ and recall the Fefferman-Stein inequality from Proposition
\ref{FeffermanSteinInequality} 
\[
\int_{\left\{ x\,:\,Mf\left(x\right)>t\right\} }\varphi\left(x\right)dx\leq\frac{C}{t}\int_{\mathbb{R}^{n}}\left|f\left(x\right)\right|M\varphi\left(x\right)dx,
\]
for any $f$ and $\varphi\geq0$. Now choose $\varphi\left(x\right)=\omega\left(x\right)\chi_{3Q}\left(x\right)$
for a fixed cube $Q$, and note that
\begin{eqnarray*}
 &  & \int_{\left\{ x\,:\,Mf\left(x\right)>t\right\} }\chi_{Q}\left(x\right)\omega\left(x\right)dx\\
 & \leq & \frac{C}{t}\int_{\mathbb{R}^{n}}\left|f\left(x\right)\right|\mathcal{M}\left(\omega\chi_{Q}\right)\left(x\right)dx\\
 & = & \frac{C}{t}\left(\int_{3Q}\left|f\left(x\right)\right|\mathcal{M}\left(\omega\chi_{Q}\right)\left(x\right)dx+\int_{3Q^{c}}\left|f\left(x\right)\right|\mathcal{M}\left(\omega\chi_{Q}\right)\left(x\right)dx\right)\\
 & = & \frac{C}{t}\left(T_{1}+T_{2}\right)\,\,\,\,\mbox{for all }t.
\end{eqnarray*}

As we did in Theorem \ref{thm:MaximalDoubling}, we will first have
a look at $T_{1}$ and in a next step at $T_{2}$. Now we estimate
the former, by using the fact, that $\omega\in A_{1}$, and hence
\[
\mathcal{M}\left(\omega\chi_{Q}\right)\left(x\right)\leq\mathcal{M}\left(\omega\right)\left(x\right)\leq C\omega\left(x\right).
\]
So after just a few steps we get 
\[
T_{1}\leq C\omega\left(3Q\right)^{\kappa}\Vert f\vert M_{u,1}\left(\omega\right)\Vert\leq C\omega\left(Q\right)^{\kappa}\Vert f\vert M_{u,1}\left(\omega\right)\Vert.
\]
For the second term $T_{2}$ we consider the form 
\[
\frac{1}{\vert R\vert}\int_{R\cap Q}\omega\left(y\right)dy
\]
for $x\in3Q^{c}$ we have $x\in R$ and $R\cap Q\neq\emptyset$. Also
recall, that $Q=Q_{r}\left(x_{0}\right)$. Geometrically speaking,
we have 
\[
\frac{1}{\vert R\vert}\int_{R\cap Q}\omega\left(y\right)dy\leq C_{n}\left(\frac{1}{\vert x-x_{0}\vert^{n}}\int_{Q}\omega\left(y\right)dy\right)\leq C_{n}\left|x-x_{0}\right|^{-n}\omega\left(Q\right).
\]
Hence $\mathcal{M}\left(\omega\chi_{Q}\right)\left(x\right)\leq C_{n}\left|x-x_{0}\right|^{-n}\omega\left(Q\right)$.
Since $\omega\in A_{1}$ we know, that it is a doubling measure with
constant $D>1$. Another estimation yields 
\begin{eqnarray*}
T_{2} & \leq & C\omega\left(Q\right)\int_{\left(3Q\right)^{c}}\frac{\vert f\left(x\right)\vert}{\vert x-x_{0}\vert^{n}}dx\\
 & \leq & C\omega\left(Q\right)\sum_{j=1}^{\infty}\frac{1}{\vert3^{j}Q\vert}\int_{3^{j+1}Q}\vert f\left(x\right)\vert dx\\
 & \leq & C\omega\left(Q\right)\sum_{j=1}^{\infty}\frac{1}{\vert3^{j}Q\vert}\frac{\vert3^{j+1}Q\vert}{\omega\left(3^{j+1}Q\right)}\int_{3^{j+1}Q}\vert f\left(x\right)\vert\omega\left(x\right)dx\\
 & = & C\omega\left(Q\right)^{\kappa}\sum_{i=1}^{\infty}\frac{\omega\left(Q\right)^{1-\kappa}}{\omega\left(3^{j+1}Q\right)^{1-\kappa}}\frac{1}{\omega\left(3^{j+1}Q\right)^{\kappa}}\int_{3^{j+1}Q}\vert f\left(x\right)\vert\omega\left(x\right)dx\\
 & \leq & C\omega\left(Q\right)^{\kappa}\Vert f\vert M_{u,1}\left(\mathbb{R}^{n},\omega\right)\Vert\sum_{j=1}^{\infty}\frac{\omega\left(Q\right)^{1-\kappa}}{\omega\left(3^{j+1}Q\right)^{1-\kappa}}\\
 & \leq & C\omega\left(Q\right)^{\kappa}\Vert f\vert M_{u,1}\left(\mathbb{R}^{n},\omega\right)\Vert
\end{eqnarray*}
Throughout this estimation we used a variety of geometric properties
and the last estimation is gained through the convergence of the series
(recall Lemma \ref{lem:LemmaDoubling}), i.e. $\frac{\omega\left(Q\right)}{\omega\left(3^{j+1}Q\right)}\leq\left(\frac{1}{D}\right)^{j+1}<1$,
since $D>1$. This completes the proof.
\end{proof}
\begin{rem}
The preceding two theorems were taken from \cite{Y. Komori; S. Shirai},
where more operators were introduced and treated the way we did with
the Hardy-Littlewood maximal operator and the weighted maximal operator.
The paper additionally deals with the Calderón-Zygmund operator 
\[
Tf\left(x\right)=\mbox{p.v.}\int_{\mathbb{R}^{n}}K\left(x-y\right)f\left(y\right)dy,
\]
where p.v. is the principle value of the integral. Furthermore for
given $0<\alpha<n$ with the fractional integral operator $I_{\alpha}$
defined by
\[
I_{\alpha}f\left(x\right)=\int_{\mathbb{R}^{n}}\frac{f\left(y\right)}{\vert x-y\vert^{n-\alpha}}dy,
\]
and with the commutator operator given by 
\[
\left[b,T\right]f\left(x\right)=b\left(x\right)Tf\left(x\right)-T\left(bf\right)\left(x\right).
\]

Here we only included the results for the maximal operators, since
they are obviously most related to Muckenhoupt weights, which we will
use for embeddings of weighted Morrey spaces. For further results
refer to \cite[Sect. 3]{Y. Komori; S. Shirai}. \newpage{}
\end{rem}

\section{\label{sec:Embeddings-of-weighted}Embeddings of Weighted Morrey
Spaces}

In this section we study whether or not the embedding

\[
M_{u_{1},q}\left(\omega\right)\hookrightarrow M_{u_{2},p}\left(\nu\right)
\]
holds or more precisely for which choice of parameters $u_{1},u_{2},q,p$
and weighs $\omega,\nu$. We will divert Muckenhoupt weights from
their originally intended use and equip Morrey spaces with them to
show that, in this case, certain embeddings can be proved. From now
on we will assume $0<q\leq u_{1}<\infty$ and $0<p\leq u_{2}$. If
$\omega$ belongs to some Muckenhoupt class, i.e. $\omega\in A_{\infty}$,
we will use the notation
\[
r_{\omega}=\inf\left\{ r\geq1\,|\,\omega\in\mathcal{A}_{r}\right\} 
\]
in this section. 
\begin{rem}
This convention is useful if we recall Corollary \ref{cor:KorollarKleinereAp}.
We showed that if $\omega\in A_{p}$ for some $p>1$, then there is
always $1<p_{1}<p$ such that $\omega\in A_{p_{1}}$, which justifies
the occurrence of the infimum term. For instance, it follows from
$\omega\in A_{1}$ that $r_{\omega}=1$. The converse, however, is
generally not true, i.e. $r_{\omega}=1$ does not imply $\omega\in A_{1}$.
Also recall the weight $\omega_{\alpha,\beta}$ from Example \ref{example1}.
We showed $\omega_{\alpha,\beta}\in A_{r}$ if and only if $-n<\alpha,\beta<n\left(r-1\right)$
or after rearranging $r>\frac{\max\left(\alpha,\beta,0\right)}{n}+1$,
and hence we have the limit case $r_{\omega_{\alpha,\beta}}=\frac{\max\left(\alpha,\beta,0\right)}{n}+1$. 
\end{rem}

What follows are some results of embeddings, if we alter $u$ and
$p$ separately.

\begin{cor}
Let $0<p_{1}\leq u_{1}<\infty$ and $0<p_{2}\leq u_{2}<\infty$. For
a given weight $\omega$ let $f\in M_{u_{1},p_{1}}\left(\omega\right)$
and $g\in M_{u_{2},p_{2}}\left(\omega\right)$. We then have $f\cdot g\in M_{u,p}\left(\omega\right)$,
if we define $\frac{1}{p}:=\frac{1}{p_{1}}+\frac{1}{p_{2}}$ and $\frac{1}{u}:=\frac{1}{u_{1}}+\frac{1}{u_{2}}$.

\end{cor}

\begin{proof}
The proof is easily deduced from Lemma \ref{lem:PreparationLemmaForuvinMup}
if we choose $f_{0}\equiv1$, $f_{1}=f$ and $f_{2}=g$.
\end{proof}
\begin{cor}
\label{cor:MuckenhouptH=0000F6lder}Let $0<p\leq q\leq u<\infty$
and $\omega$ be a weight, then we have 
\[
M_{u,q}\left(\omega\right)\hookrightarrow M_{u,p}\left(\omega\right).
\]
\end{cor}

\begin{proof}
We use Lemma \ref{lem:PreparationLemmaForuvinMup} with the choice
of $n=1$ and $f_{0}=1$. For $f\in M_{u,q}\left(\omega\right)$ we
quickly get that $f\in M_{u,p}\left(\omega\right)$ if we choose $\frac{1}{q}\leq\frac{1}{p}$
or $p\leq q$ respectively. The continuity of the embedding can be
seen in the last steps of the proof of the lemma. 
\end{proof}
Naturally Lemma \ref{Re: RepeatedRemark} (iii) also hints that (just
like their unweighted counterparts) weighted Morrey spaces are a generalization
of weighted Lebesgue spaces. However, there are functions that exclusively
belong to weighted Morrey spaces. We refer to \cite[Rem. 2.3. (3)]{Y. Komori; S. Shirai}
for this example:\newpage{}
\begin{example}
For this example let $n=1$ and $\omega\left(x\right)=\left|x\right|^{\alpha}$
for some negative parameter $-\frac{1}{2}<\alpha<0$. By Example \ref{example1}
we know by the choice of parameters that $\omega\in A_{p}$. If we
have a look at the function $f\left(x\right)=\chi_{\left(0,1\right)}\left|x\right|^{-\frac{1}{2}}$,
we can indeed verify that
\[
f\in M_{2\left(\alpha+1\right),1}\left(\omega\right)\backslash M_{2\left(\alpha+1\right),2\left(\alpha+1\right)}\left(\omega\right)=M_{2\left(\alpha+1\right),1}\left(\omega\right)\backslash L_{2\left(\alpha+1\right)}\left(\omega\right)
\]
By Corollary \ref{cor:MuckenhouptH=0000F6lder} we know that $M_{2\left(\alpha+1\right),q}\left(\omega\right)\hookrightarrow M_{2\left(\alpha+1\right),1}\left(\omega\right)$
for $1<q\leq2\left(\alpha+1\right)$, such that $f$ exclusively belongs
to weighted Morrey spaces.
\end{example}

\begin{cor}
\label{cor:MuckenhouptH=0000F6lder2}Let $0<q\leq u_{1}<\infty$ and
$\omega$ be a weight. If $u_{1}=u_{2}$ and $p\leq q$, then
\[
M_{u_{1},q}\left(\omega\right)\hookrightarrow M_{u_{2},p}\left(\omega\right).
\]

.
\end{cor}

\begin{proof}
We have already seen this statement before in Corollary \ref{cor:MuckenhouptH=0000F6lder},
although we did not try to alter $u$ there. Let us have a look at
an alternative proof technique. For $f\in M_{u_{1},q}\left(\omega\right)$
and $B$ a ball we see the following 
\begin{eqnarray*}
 &  & \omega\left(B\right)^{1/u_{2}-1/p}\left(\int_{B}\left|f\left(x\right)\right|^{p}\omega\left(x\right)dx\right)^{1/p}\\
 & = & \omega\left(B\right)^{1/u_{2}-1/p}\left(\int_{B}\left|f\left(x\right)\right|^{p}\omega\left(x\right)^{p/q}\omega\left(x\right)^{-p/q}\omega\left(x\right)dx\right)^{1/p}\\
 & \leq & \omega\left(B\right)^{1/u_{2}-1/p}\left(\int_{B}\left|f\left(x\right)\right|^{q}\omega\left(x\right)dx\right)^{1/q}\left(\int_{B}\omega\left(x\right)^{-\frac{p}{q-p}}\omega\left(x\right)^{\frac{q}{q-p}}dx\right)^{1/p-1/q}\\
 & = & \omega\left(B\right)^{1/u_{2}-1/p+1/p-1/q+1/u_{1}-1/u_{1}}\left(\int_{B}\left|f\left(x\right)\right|^{q}\omega\left(x\right)dx\right)^{1/q}\\
 & = & \omega\left(B\right)^{1/u_{2}-1/u_{1}}\omega\left(B\right)^{1/u_{1}-1/q}\left(\int_{B}\left|f\left(x\right)\right|^{q}\omega\left(x\right)dx\right)^{1/q}\\
 & \leq & \sup_{B}\omega\left(B\right)^{1/u_{2}-1/u_{1}}\Vert f\vert M_{u,q}\left(\omega\right)\Vert,
\end{eqnarray*}
where we used Hölder's inequality for $p<q$. Now if we assume that
$u_{1}=u_{2}$ the last estimate obviously holds.
\end{proof}
\begin{rem}
Corollary \ref{cor:MuckenhouptH=0000F6lder2} does not yield any new
information when we compare it to Corollary \ref{cor:MuckenhouptH=0000F6lder}.
It uses $u_{1}=u_{2}$ as a sufficient condition for the embedding
to hold and this different proof technique leaves us with no evidence
that it could also be a necessary condition. However, considering
the weight $\omega\equiv1$ (which classifies for the conditions in
Corollary \ref{cor:MuckenhouptH=0000F6lder2}) brings us back to Theorem
\ref{thm:Sch=0000F6nesTheorem} and gives us an idea that $u_{1}=u_{2}$
could indeed be mandatory.\newpage{}
\end{rem}

\begin{prop}
\label{prop:GewichtInUngewicht}Let $0<q\leq u_{1}<\infty$ and $0<p\leq u_{2}<\infty$.
Furthermore let $\omega\in A_{\infty}$ with $r_{\omega}>1$ and the
following two conditions be true
\begin{equation}
\sup_{B}\frac{\left|B\right|^{1/u_{2}}}{\omega\left(B\right)^{1/u_{1}}}\leq C<\infty\label{eq:GewichtInUngewichtErsteBedingung}
\end{equation}
\begin{equation}
r_{\omega}<\frac{q}{p}.\label{eq:GewichtInUngewichtZweiteBedingung}
\end{equation}
Then the following embedding holds
\[
M_{u_{1},q}\left(\omega\right)\hookrightarrow M_{u_{2},p}.
\]
Furthermore, if $\omega\in A_{1}$ then $p=q$ is admitted.
\end{prop}

\begin{proof}
Let $f\in M_{u_{1},q}\left(\omega\right)$ and $B$ be a ball, then
\begin{eqnarray*}
\left(\int_{B}\left|f\left(x\right)\right|^{p}dx\right)^{1/p} & = & \left(\int_{B}\left|f\left(x\right)\right|^{p}\omega\left(x\right)^{p/q}\omega\left(x\right)^{-p/q}dx\right)^{1/p}\\
 & \leq & \left(\int_{B}\left|f\left(x\right)\right|^{q}\omega\left(x\right)dx\right)^{1/q}\left(\int_{B}\omega\left(x\right)^{-\frac{p}{q-p}}dx\right)^{1/p-1/q}\\
 & = & \omega\left(B\right)^{1/u_{1}-1/q}\left(\int_{B}\left|f\left(x\right)\right|^{q}\omega\left(x\right)dx\right)^{1/q}\\
 &  & \omega\left(B\right)^{1/q-1/u_{1}}\left(\int_{B}\omega\left(x\right)^{-\frac{p}{q-p}}dx\right)^{1/p-1/q}.
\end{eqnarray*}
Here we used Hölder's inequality in the first estimation, which holds,
since $\frac{q}{p}>r_{\omega}>1$ by assumption \eqref{eq:GewichtInUngewichtZweiteBedingung}.
After multiplying both sides of the inequality with $\left|B\right|^{1/u_{2}-1/p}$,
we see that $f\in M_{u_{2},p}$ if 
\begin{equation}
\sup_{B}\left|B\right|^{1/u_{2}-1/p}\omega\left(B\right)^{1/q-1/u_{1}}\left(\int_{B}\omega\left(x\right)^{-\frac{p}{q-p}}dx\right)^{1/p-1/q}<\infty.\label{eq:MuckenhouptInAction1}
\end{equation}
Now by assumption \eqref{eq:GewichtInUngewichtZweiteBedingung} $r_{\omega}<\frac{q}{p}$,
we know that $\omega\in A_{\frac{q}{p}}$ and hence
\begin{equation}
\left(\frac{1}{\left|B\right|}\int_{B}\omega\left(x\right)dx\right)\left(\frac{1}{\left|B\right|}\int_{B}\omega\left(x\right)^{-\frac{p}{q-p}}dx\right)^{\frac{q-p}{p}}\leq C.\label{eq:MuckenhouptInAction100}
\end{equation}
 holds for some $C>0$. So after taking \eqref{eq:MuckenhouptInAction100}
to the $1/q$-th power, we see that \eqref{eq:MuckenhouptInAction1}
yields
\begin{eqnarray*}
 &  & \left|B\right|^{1/u_{2}-1/p}\omega\left(B\right)^{1/q-1/u_{1}}\left(\int_{B}\omega\left(x\right)^{-\frac{p}{q-p}}dx\right)^{1/p-1/q}\\
 & = & \frac{\left|B\right|^{1/u_{2}}}{\omega\left(B\right)^{1/u_{1}}}\underset{\leq C}{\underbrace{\left|B\right|^{-1/p}\omega\left(B\right)^{1/q}\left(\int_{B}\omega\left(x\right)^{-\frac{p}{q-p}}dx\right)^{1/p-1/q}}}.
\end{eqnarray*}
After taking the supremum over all such balls $B$, we see that $\sup_{B}\frac{\left|B\right|^{1/u_{2}}}{\omega\left(B\right)^{1/u_{1}}}\leq C<\infty$
by assumption \eqref{eq:GewichtInUngewichtErsteBedingung}. 

We will now have a look at the case $\omega\in A_{1}$ in which $p=q$
is allowed. We have
\begin{eqnarray}
 &  & \left|B\right|^{1/u_{2}-1/p}\left(\int_{B}\left|f\left(x\right)\right|^{p}dx\right)^{1/p}\nonumber \\
 & = & \left|B\right|^{1/u_{2}-1/q}\left(\int_{B}\left|f\left(x\right)\right|^{q}\omega\left(x\right)\omega\left(x\right)^{-1}dx\right)^{1/q}\nonumber \\
 & \leq & \Vert\omega^{-1}\vert L_{\infty}\left(B\right)\Vert^{1/q}\left|B\right|^{1/u_{2}-1/q}\left(\int_{B}\left|f\left(x\right)\right|^{q}\omega\left(x\right)dx\right)^{1/q}\nonumber \\
 & = & \Vert\omega^{-1}\vert L_{\infty}\left(B\right)\Vert^{1/q}\omega\left(B\right)^{1/q-1/u_{1}}\left|B\right|^{1/u_{2}-1/q}\label{eq:SpecialCase1}\\
 &  & \cdot\omega\left(B\right)^{1/u_{1}-1/q}\left(\int_{B}\left|f\left(x\right)\right|^{q}\omega\left(x\right)dx\right)^{1/q}.\nonumber 
\end{eqnarray}
Since $\omega\in A_{1}$, we conclude $\Vert\omega^{-1}\vert L_{\infty}\left(B\right)\Vert\leq A\left|B\right|\left(\int_{B}\omega\left(x\right)dx\right)^{-1}$
for all balls $B$. Taking this inequality to the $1/q$-th power
and inserting it in \eqref{eq:SpecialCase1} leaves us with
\begin{eqnarray*}
 &  & \frac{\left|B\right|^{1/u_{2}}}{\omega\left(B\right)^{1/u_{1}}}\omega\left(B\right)^{1/u_{1}-1/q}\left(\int_{B}\left|f\left(x\right)\right|^{q}\omega\left(x\right)dx\right)^{1/q}\\
 & \leq & \sup_{B}\frac{\left|B\right|^{1/u_{2}}}{\omega\left(B\right)^{1/u_{1}}}\Vert f\vert M_{u_{1},q}\left(\omega\right)\Vert,
\end{eqnarray*}
where the first term is finite by assumption.
\end{proof}
\begin{example}
Consider $\omega\equiv1$, then $\omega\in A_{s}$ for every $1\leq s\leq\infty$.
So for $\omega\equiv1$, condition \eqref{eq:GewichtInUngewichtErsteBedingung}
becomes $\sup_{B}\left|B\right|^{1/u_{2}-1/u_{1}}\leq C<\infty$ and
hence necessarily $u_{1}=u_{2}$. Condition \eqref{eq:GewichtInUngewichtZweiteBedingung}
then turns into $p\leq q$. Through this choice of parameters coincides
with Theorem \ref{thm:Sch=0000F6nesTheorem} or Corollary \ref{cor:Sch=0000F6nesKorollar}
respectively and is what we expected for unweighted Morrey spaces.
\end{example}

\begin{rem}
Note that the requirement $0<p\leq u_{2}$ is not explicitly needed
if we additionally demand $\sup_{B_{1}}\omega\left(B_{1}\right)\leq A<\infty$
for balls of radius $1$, (i.e. $B_{1}$), but is rather implied by
$\sup_{B}\frac{\left|B\right|^{1/u_{2}}}{\omega\left(B\right)^{1/u_{1}}}\leq C<\infty$
and $0<q\leq u_{1}$. To see this, we have a look at all balls $B$
with $R>1$ and recall
\[
\frac{\left|B_{1}\right|}{\left|B_{R}\right|}\leq C\left(\frac{\omega\left(B_{1}\right)}{\omega\left(B_{R}\right)}\right)^{1/r}\leq CA^{1/r}\cdot\omega\left(B_{R}\right)^{-1/r},
\]
for some $r>1$ (see Remark \ref{Remark: AboutDuo} and also \cite[Ch. V, 1.7]{E.M. Stein})
from which immediately follows that
\[
\frac{1}{\omega\left(B_{R}\right)}\geq\frac{c}{\left|B_{R}\right|^{r}}.
\]
Now the chain of inequalities is as follows:
\[
\infty>C\geq\sup_{B}\frac{\left|B\right|^{1/u_{2}}}{\omega\left(B\right)^{1/u_{1}}}\geq\sup_{B:\left|B\right|>1}\frac{\left|B\right|^{1/u_{2}}}{\omega\left(B\right)^{1/u_{1}}}\geq c\sup_{B:\left|B\right|>1}\left|B\right|^{1/u_{2}-r/u_{1}}.
\]
Naturally we need $\frac{1}{u_{2}}-\frac{r}{u_{1}}\leq0$ and with
condition \eqref{eq:GewichtInUngewichtZweiteBedingung} of Proposition
\ref{prop:GewichtInUngewicht} we get $\frac{u_{1}}{u_{2}}\leq r\leq\frac{q}{p}$.
This turns into $\frac{p}{u_{2}}\leq\frac{q}{u_{1}}\leq1$ by assumption
and thus $p\leq u_{2}$.
\end{rem}

\begin{cor}
\label{cor: ExampleWeight1} Let $-n<\alpha,\beta<n\left(p-1\right)$
and $\omega\left(x\right)=\omega_{\alpha,\beta}\left(x\right)=\begin{cases}
\left|x\right|^{\alpha} & ,\,\left|x\right|\leq1\\
\left|x\right|^{\beta} & ,\,\left|x\right|>1
\end{cases}$. Assume
\begin{equation}
\frac{1}{u_{1}}\left(1+\frac{\beta}{n}\right)\geq\frac{1}{u_{2}}\geq\frac{1}{u_{1}}\left(1+\frac{\max\left(\alpha,0\right)}{n}\right)\,\mbox{ and}\label{eq:ErstesBeispielErsteBedingung}
\end{equation}
\begin{equation}
\frac{1}{p}\geq\frac{1}{q}\left(1+\frac{\max\left(\alpha,\beta\right)}{n}\right).\label{eq:ErstesBeispielZweiteBedingung}
\end{equation}
Then the following embedding holds
\[
M_{u_{1},q}\left(\omega_{\alpha,\beta}\right)\hookrightarrow M_{u_{2},p}.
\]
\end{cor}

\begin{proof}
We want to apply Proposition \ref{prop:GewichtInUngewicht} and therefore
need to check conditions \eqref{eq:GewichtInUngewichtErsteBedingung}
and \eqref{eq:GewichtInUngewichtZweiteBedingung} for the specific
weight $\omega_{\alpha,\beta}$. We already know that $\omega_{\alpha,\beta}\in A_{r}$
if and only if $-n<\alpha,\beta<n\left(r-1\right)$ (see Example \ref{example1}).

(i) At first we will have a look at the requirement $\sup_{B}\frac{\left|B\right|^{1/u_{2}}}{\omega\left(B\right)^{1/u_{1}}}\leq C<\infty$.
Since all balls $B$ are admissible for taking the supremum, we have
to differ between all the possible cases. The different parts we will
consider, are similar to the ones we checked earlier on in Example
\ref{example1}. For all considered cases, we keep in mind that the
parameters are set as $-n<\alpha,\beta<n\left(p-1\right)$.\\

\emph{Part 1} : Let $x_{0}=0$. Then for $R<1$
\[
\sup_{B}\frac{\left|B\right|^{1/u_{2}}}{\omega\left(B\right)^{1/u_{1}}}\leq C\sup_{R<1}R^{n/u_{2}-\left(\alpha+n\right)/u_{1}},
\]
which we find to be finite for $u_{2}\leq$$\frac{nu_{1}}{\left(\alpha+n\right)}$.
For $R>1$ we are in the situation
\[
\sup_{B}\frac{\left|B\right|^{1/u_{2}}}{\omega\left(B\right)^{1/u_{1}}}\leq C\sup_{R>1}R^{n/u_{2}-\left(\beta+n\right)/u_{1}},
\]
where we need $u_{2}\geq\frac{nu_{1}}{\left(\beta+n\right)}$. \newpage{}

\emph{Part 2 }: Now let $R>0$ and $0<\left|x_{0}\right|<\frac{R}{2}$. 

As before the following inclusion holds: $B_{\frac{R}{2}}\left(0\right)\subset B_{R}\left(x_{0}\right)\subset B_{\frac{3}{2}R}\left(0\right)$.
Also, we have for some $c,C>0$ that $c\left|B_{\frac{3}{2}R}\left(0\right)\right|\leq\left|B_{R}\left(x_{0}\right)\right|\leq C\left|B_{\frac{R}{2}}\left(0\right)\right|$,
from which (with the help of \eqref{eq:BallIntegralUngleichung})
we now deduce
\begin{eqnarray*}
 &  & c\left|B_{\frac{3}{2}R}\left(0\right)\right|^{1/u_{2}}\left(\int_{B_{\frac{3}{2}R}\left(0\right)}\omega_{\alpha,\beta}\left(x\right)dx\right)^{-1/u_{1}}\\
 & \leq & \left|B_{R}\left(x_{0}\right)\right|^{1/u_{2}}\left(\int_{B_{R}\left(x_{0}\right)}\omega_{\alpha,\beta}\left(x\right)dx\right)^{-1/u_{1}}\\
 & \leq & C\left|B_{\frac{1}{2}R}\left(0\right)\right|^{1/u_{2}}\left(\int_{B_{\frac{1}{2}R}\left(0\right)}\omega_{\alpha,\beta}\left(x\right)dx\right)^{-1/u_{1}}.
\end{eqnarray*}
For $R<\frac{2}{3}$ or $R>2$ we get $\frac{1}{2}R<\frac{3}{2}R<1$
and $\frac{3}{2}R>\frac{R}{2}>1$ respectively, meaning that we can
apply part 1 to get the same conditions for $u_{2}$. For the remaining
case $\frac{2}{3}<R<2$, we go back to \eqref{eq:Absch=0000E4tzungObenExample}
and \eqref{eq:Absch=0000E4tzungUntenExample}. Together with the estimation
$c\left(\frac{2}{3}\right)^{n}\leq\left|B_{R}\left(x_{0}\right)\right|\leq C2^{n}$
we get that $\left|B_{R}\left(x_{0}\right)\right|^{1/u_{2}}\left(\int_{B_{R}\left(x_{0}\right)}\omega_{\alpha,\beta}\left(x\right)dx\right)^{1/u_{1}}\sim C$,
which gives us no new information about the choice of values for $u_{2}$.\\

\emph{Part 3} : Let $R>0$ and $\left|x_{0}\right|>2R$.

In this case, we use the triangle inequality for all $x\in B_{R}\left(x_{0}\right)$
to get 
\begin{eqnarray}
\left|x\right| & \leq & \left|x_{0}\right|+\left|x-x_{0}\right|\leq\frac{3}{2}\left|x_{0}\right|\,\mbox{and}\label{eq:Triangle1}\\
\left|x\right| & \geq & \left|x_{0}\right|-\left|x-x_{0}\right|\geq\frac{1}{2}\left|x_{0}\right|,\label{eq:Triangle2}
\end{eqnarray}
which is essentially $\omega_{\alpha,\beta}\left(x_{0}\right)\sim\omega_{\alpha,\beta}\left(x\right)$
for all $x\in B_{R}\left(x_{0}\right)$. We subdivide part 3 into
smaller steps, namely $\left|x_{0}\right|<\frac{2}{3}$, $\frac{2}{3}<\left|x_{0}\right|<2$
and $2<\left|x_{0}\right|$.

Let us first check $\left|x_{0}\right|>2$, for which we get $\left|x\right|>1$
for all $x\in B_{R}\left(x_{0}\right)$ by \eqref{eq:Triangle2} .
We have to find an assignment of parameters, such that
\begin{eqnarray*}
\sup_{\left|x_{0}\right|>2}\sup_{0<R\leq1}\left|B_{R}\left(x_{0}\right)\right|^{1/u_{2}}\omega_{\alpha,\beta}\left(B_{R}\left(x_{0}\right)\right)^{-1/u_{1}} & \leq & A<\infty\mbox{ and}\\
\sup_{R>1}\sup_{\left|x_{0}\right|>2R}\left|B_{R}\left(x_{0}\right)\right|^{1/u_{2}}\omega_{\alpha,\beta}\left(B_{R}\left(x_{0}\right)\right)^{-1/u_{1}} & \leq & A<\infty
\end{eqnarray*}
hold. Now since $\int_{B}\omega_{\alpha,\beta}\left(x\right)dx\sim\int_{B}\omega_{\alpha,\beta}\left(x_{0}\right)dx$=$\left|B\right|\omega_{\alpha,\beta}\left(x_{0}\right)$,
the above inequalities are equivalent to 
\begin{eqnarray}
\sup_{\left|x_{0}\right|>2}\sup_{0<R\leq1}R^{n\left(1/u_{2}-1/u_{1}\right)}\left|x_{0}\right|^{-\beta/u_{1}} & \leq & A<\infty\mbox{ and}\label{eq:Triangle3}\\
\sup_{R>1}\sup_{\left|x_{0}\right|>2R}R^{n\left(1/u_{2}-1/u_{1}\right)}\left|x_{0}\right|^{-\beta/u_{1}} & \leq & A<\infty.\label{eq:Triangle4}
\end{eqnarray}
Inequality \eqref{eq:Triangle3} now requires $u_{1}\geq u_{2}$ and
$\beta\geq0$ to be bounded whereas \eqref{eq:Triangle4} leads to
\[
\sup_{R>1}R^{n\left(1/u_{2}-1/u_{1}\right)-\beta/u_{1}}\leq A<\infty,
\]
from which we again get the condition $\frac{\beta+n}{nu_{1}}\geq\frac{1}{u_{2}}$. 

We move on to the next case, namely $\left|x_{0}\right|<\frac{2}{3}$,
for which we get $\left|x\right|<1$ for all $x\in B_{R}\left(x_{0}\right)$
by \eqref{eq:Triangle1}. Similarly to the above, we need to check
for which choice of parameters 
\begin{equation}
\sup_{0<R<\frac{1}{3}}\sup_{2R<\left|x_{0}\right|<\frac{2}{3}}R^{n\left(1/u_{2}-1/u_{1}\right)}\left|x_{0}\right|^{-\alpha/u_{1}}\leq A<\infty.\label{eq:MaximumCaseForBalls}
\end{equation}
As we can see, this is only finite, if $\frac{1}{u_{2}}\geq\frac{\max\left(\alpha,0\right)+n}{nu_{1}}$.
The maximum occurs, because we are considering $\left|x_{0}\right|\leq\frac{2}{3}$,
which are bounded for negative $\alpha$ anyway.

The last remaining case is $\frac{2}{3}\leq\left|x_{0}\right|\leq2$,
which converts to $\frac{1}{3}\leq\left|x\right|\leq3$ for all $x\in B_{R}\left(x_{0}\right)$.
Here we rather quickly get
\[
\min\left(\left(\frac{1}{3}\right)^{\alpha},1,3^{\beta}\right)\leq\omega_{\alpha,\beta}\left(x\right)\leq\max\left(\left(\frac{1}{3}\right)^{\alpha},1,3^{\beta}\right),
\]
and thus $\left|B_{R}\left(x_{0}\right)\right|^{1/u_{2}}\omega_{\alpha,\beta}\left(B_{R}\left(x_{0}\right)\right)^{-1/u_{1}}\sim C$
in this case.\\

\emph{Part 4: }Let \emph{$\frac{1}{2}R\leq\left|x_{0}\right|\leq2R$.}

We reduce this last remaining part to the previous ones in two steps.
At first recall the triangle inequality once more
\[
\left|x\right|\leq\left|x_{0}\right|+\left|x-x_{0}\right|\le3R,\mbox{ for all }x\in B_{R}\left(x_{0}\right),
\]
i.e. $B_{R}\left(x_{0}\right)\subset B_{3R}\left(0\right)$, which
gives an upper bound and brings us back to Part 1. As for the second
step, recall that $B_{R}\left(x_{0}\right)$ is an open set, thus
enclosing another open ball with smaller radius, say $R'=\frac{1}{6}R$.
With $\left|x_{0}\right|\geq\frac{1}{2}R>\frac{1}{3}R=2R'$ and $B_{R'}\left(x_{0}\right)\subset B_{R}\left(x_{0}\right)$,
we get a lower bound and return to Part 3. Thus Part 4 does not yield
any new relevant information about any of the parameters. As a conclusion
we get \eqref{eq:ErstesBeispielErsteBedingung}.

(ii) As we have already seen in Example \ref{example1}, we know that
$\omega_{\alpha,\beta}\in A_{r}$ if and only if $-n<\alpha,\beta<n\left(r-1\right)$.
Since we are dealing with Muckenhoupt classes, we are interested in
$r\geq1$, so by rearranging the previous inequality, we find $r>\frac{\max\left(\alpha,\beta\right)}{n}+1$
to be a valid requirement (note that \eqref{eq:ErstesBeispielErsteBedingung}
indirectly implies $\beta\geq0$.). Lastly we have $r_{\omega_{\alpha,\beta}}=\frac{\max\left(\alpha,\beta\right)}{n}+1$.
This finally yields \eqref{eq:ErstesBeispielZweiteBedingung} and
concludes the proof.
\end{proof}
\newpage{}
\begin{prop}
\label{prop:UngewichtInGewicht}Let $0<q\leq u_{1}$ and $0<p\leq u_{2}$.
Furthermore let $\nu^{-1}\in A_{\infty}$ and the following two conditions
be true:
\begin{equation}
\sup_{B}\left(\frac{\left|B\right|}{\nu^{-1}\left(B\right)}\right)^{1/p}\left(\frac{\nu\left(B\right)}{\left|B\right|}\right)^{1/u_{2}-1/p}\left|B\right|^{1/u_{2}-1/u_{1}}\leq C<\infty\label{eq:UngewichtInGewichtErsteBedingung}
\end{equation}
\begin{equation}
r_{\nu^{-1}}<2-\frac{p}{q}.\label{eq:UngewichtInGewichtZweiteBedingung}
\end{equation}
Then the following embedding holds
\[
M_{u_{1},q}\hookrightarrow M_{u_{2},p}\left(\nu\right).
\]
Furthermore, if $\nu^{-1}\in A_{1}$ then $p=q$ is admitted.
\end{prop}

\begin{proof}
Let $f\in M_{u_{1},q}$ then by using Hölder's inequality with $\frac{q}{p}>1$
we gain
\begin{eqnarray*}
 &  & \nu\left(B\right)^{1/u_{2}-1/p}\left(\int_{B}\left|f\left(x\right)\right|^{p}\nu\left(x\right)dx\right)^{1/p}\\
 & \leq & \nu\left(B\right)^{1/u_{2}-1/p}\left(\int_{B}\left|f\left(x\right)\right|^{q}dx\right)^{1/q}\left(\int_{B}\nu\left(x\right)^{\frac{q}{q-p}}dx\right)^{\frac{q-p}{qp}}\\
 & = & \left|B\right|^{1/u_{1}-1/q}\left(\int_{B}\left|f\left(x\right)\right|^{q}dx\right)^{1/q}\left|B\right|^{1/q-1/u_{1}}\\
 &  & \cdot\nu\left(B\right)^{1/u_{2}-1/p}\left(\int_{B}\nu\left(x\right)^{\frac{q}{q-p}}dx\right)^{\frac{q-p}{qp}}.
\end{eqnarray*}
We see that $f\in M_{u_{2},p}\left(\nu\right)$ if 
\begin{equation}
\sup_{B}\left|B\right|^{1/q-1/u_{1}}\nu\left(B\right)^{1/u_{2}-1/p}\left(\int_{B}\nu\left(x\right)^{\frac{q}{q-p}}dx\right)^{\frac{q-p}{qp}}\leq C<\infty.\label{eq:ConditionUngewichtInGewicht}
\end{equation}
Now similarly to Proposition \ref{prop:GewichtInUngewicht}, we want
to use the Muckenhoupt properties of $\nu^{-1}$. By assumption \eqref{eq:UngewichtInGewichtZweiteBedingung},
i.e. $r_{\nu^{-1}}<2-\frac{p}{q}$, we know that $\nu^{-1}\in A_{2-\frac{p}{q}}$
and thus
\[
\left(\frac{1}{\left|B\right|}\int_{B}\nu^{-1}\left(x\right)dx\right)\left(\frac{1}{\left|B\right|}\int_{B}\left(\nu^{-1}\right)^{-\frac{q}{q-p}}dx\right)^{\frac{q-p}{q}}\leq C<\infty.
\]
Of course this inequality holds after taking it to the $1/p$-th power.
Now we see that \eqref{eq:ConditionUngewichtInGewicht} turns into
\begin{eqnarray*}
 &  & \left|B\right|^{1/q-1/u_{1}}\nu\left(B\right)^{1/u_{2}-1/p}\left(\int_{B}\nu\left(x\right)^{\frac{q}{q-p}}dx\right)^{\frac{q-p}{qp}}\\
 & \leq & \left|B\right|^{1/q-1/u_{1}}\nu\left(B\right)^{1/u_{2}-1/p}\left|B\right|^{1/p}\nu^{-1}\left(B\right)^{-1/p}\\
 &  & \cdot\left|B\right|^{\frac{q-p}{qp}}\left(\frac{1}{\left|B\right|}\nu^{-1}\left(B\right)\right)^{1/p}\left(\frac{1}{\left|B\right|}\int_{B}\nu^{\frac{q}{q-p}}dx\right)^{\frac{q-p}{qp}}\\
 & \leq & C\left|B\right|^{1/q-1/u_{1}}\cdot\nu\left(B\right)^{1/u_{2}-1/p}\left|B\right|^{1/p}\\
 &  & \cdot\nu^{-1}\left(B\right)^{-1/p}\left|B\right|^{1/p-1/q}\\
 & = & C\left(\frac{\left|B\right|}{\nu^{-1}\left(B\right)}\right)^{1/p}\left(\frac{\nu\left(B\right)}{\left|B\right|}\right)^{1/u_{2}-1/p}\left|B\right|^{1/u_{2}-1/u_{1}}.
\end{eqnarray*}
By assumption the last term is bounded after taking the supremum over
all balls $B$ and thus $f\in M_{u_{2},p}\left(\nu\right)$. 

At last, let us have a look at the case $\nu^{-1}\in A_{1}$ in which
$p=q$ is allowed. Let $f\in M_{u_{1},q}$ to see
\begin{eqnarray}
 &  & \nu\left(B\right)^{1/u_{2}-1/p}\left(\int_{B}\left|f\left(x\right)\right|^{p}\nu\left(x\right)dx\right)^{1/p}\nonumber \\
 & = & \nu\left(B\right)^{1/u_{2}-1/p}\left(\int_{B}\left|f\left(x\right)\right|^{q}\nu\left(x\right)dx\right)^{1/q}\nonumber \\
 & \leq & \nu\left(B\right)^{1/u_{2}-1/p}\Vert\nu\vert L_{\infty}\left(B\right)\Vert\left(\int_{B}\left|f\left(x\right)\right|^{q}\nu\left(x\right)dx\right)^{1/q}\nonumber \\
 & = & \nu\left(B\right)^{1/u_{2}-1/p}\Vert\nu\vert L_{\infty}\left(B\right)\Vert\left|B\right|^{1/p-1/u_{1}}\nonumber \\
 &  & \cdot\left|B\right|^{1/u_{1}-1/q}\left(\int_{B}\left|f\left(x\right)\right|^{q}\nu\left(x\right)dx\right)^{1/q}.\label{eq:EquationBarsch}
\end{eqnarray}
Now since $\nu^{-1}\in A_{1}$, we know that $\frac{1}{\left|B\right|}\int_{B}\nu\left(x\right)^{-1}dx\leq C\nu\left(y\right)^{-1}$
for all $y\in B$. Rearranging and taking the $1/p$-th power yields
$C\left(\nu\left(y\right)^{-1}\right)^{-1/p}\leq\left|B\right|^{1/p}\left(\int_{B}\nu\left(x\right)^{-1}dx\right)^{-1/p}$.
Thus taking the supremum over all balls $B$ in \eqref{eq:EquationBarsch}
leaves us with 
\begin{eqnarray*}
 &  & \sup_{B}\nu\left(B\right)^{1/u_{2}-1/p}\Vert\nu\vert L_{\infty}\left(B\right)\Vert\left|B\right|^{1/p-1/u_{1}}\Vert f\vert M_{u_{1},q}\Vert\\
 & \leq & \sup_{B}\left(\frac{\left|B\right|}{\nu^{-1}\left(B\right)}\right)^{1/p}\left(\frac{\nu\left(B\right)}{\left|B\right|}\right)^{1/u_{2}-1/p}\left|B\right|^{1/u_{2}-1/u_{1}}\Vert f\vert M_{u_{1},q}\Vert,
\end{eqnarray*}
which is finite by assumption \eqref{eq:UngewichtInGewichtErsteBedingung}.
This concludes the proof.
\end{proof}
\newpage{}
\begin{example}
Let us again take the distinctive example $\nu\equiv1$ as an application
for Proposition \ref{prop:UngewichtInGewicht}. By doing so, we see
that the first condition shrinks down to $\sup_{B}\left|B\right|^{1/u_{2}-1/u_{1}}\leq C<\infty$,
which indeed can only be satisfied if and only if $u_{1}=u_{2}$.
Furthermore $\nu\equiv1\in A_{1}$ and thus $p\leq q$ is allowed,
hence we return to Corollary \ref{cor:Sch=0000F6nesKorollar} again.
\end{example}

\begin{cor}
\label{cor:ExampleWeight2}Let $\left|\alpha\right|,\left|\beta\right|<n$
and $\omega\left(x\right)=\omega_{\alpha,\beta}\left(x\right)=\begin{cases}
\left|x\right|^{\alpha} & ,\,\left|x\right|\leq1\\
\left|x\right|^{\beta} & ,\,\left|x\right|>1
\end{cases}$. \\
Assume
\begin{equation}
\frac{1}{u_{2}}\left(1+\frac{\beta}{n}\right)\leq\frac{1}{u_{1}}\leq\frac{1}{u_{2}}\left(1+\frac{\min\left(\alpha,0\right)}{n}\right)\,\mbox{ and}\label{eq:ZweitesBeispielErsteBedingung}
\end{equation}
\begin{equation}
\frac{1}{q}\leq\frac{1}{p}\left(1+\frac{\min\left(\alpha,\beta\right)}{n}\right).\label{eq:ZweitesBeispielZweiteBedingung}
\end{equation}
Then the following embedding holds
\[
M_{u_{1},q}\hookrightarrow M_{u_{2},p}\left(\omega_{\alpha,\beta}\right).
\]
\end{cor}

\begin{proof}
To prove the statement we apply Proposition \ref{prop:UngewichtInGewicht}
for the example weight $\omega_{\alpha,\beta}$, and need to find
parameters $u_{1}$ and $q$ such that conditions \eqref{eq:UngewichtInGewichtErsteBedingung}
and \eqref{eq:UngewichtInGewichtZweiteBedingung} hold. 

(i) To verify the first condition, we have to cover all the different
types of balls, as we have already seen in Example \ref{example1}
or Corollary \ref{cor: ExampleWeight1}. This time, we simply reduce
the proof to the just mentioned corollary, since the cases are the
same, except for some more terms in the supremum, which we will take
care of as follows:

The balls of main interest are the ones centered at the origin and
radius $R<1$ and $R>1$ respectively, since we are trying to reduce
every other case to these two. So let us have a look at balls $B_{R}\left(0\right)$
with $R<1$ as an example. We recall the equivalence $\omega_{\alpha,\beta}\left(B\right)\sim R^{n+\alpha}\sim\left|B\right|^{1+\alpha/n}$.
Also observe that $\omega_{\alpha,\beta}^{-1}\left(x\right)=\omega_{-\alpha,-\beta}\left(x\right)$
and consequently $\omega_{\alpha,\beta}^{-1}\left(B_{R}\left(0\right)\right)\sim\left|B_{R}\left(0\right)\right|^{1-\alpha/n}$.
With that in mind, we see
\begin{eqnarray*}
 &  & \sup_{B_{R}\left(0\right),R<1}\left(\frac{\left|B\right|}{\omega_{\alpha,\beta}^{-1}\left(B\right)}\right)^{1/p}\left(\frac{\omega_{\alpha,\beta}\left(B\right)}{\left|B\right|}\right)^{1/u_{2}-1/p}\left|B\right|^{1/u_{2}-1/u_{1}}\\
 & \sim & \sup_{B_{R}\left(0\right),R<1}\left(\frac{\left|B\right|}{\left|B\right|^{1-\frac{\alpha}{n}}}\right)^{1/p}\left(\frac{\left|B\right|^{1+\frac{\alpha}{n}}}{\left|B\right|}\right)^{1/u_{2}-1/p}\left|B\right|^{1/u_{2}-1/u_{1}}\\
 & = & \sup_{B_{R}\left(0\right),R<1}\left|B\right|^{\frac{1}{u_{2}}\left(\frac{\alpha}{n}+1\right)-\frac{1}{u_{2}}}<\infty
\end{eqnarray*}
which is true if $\frac{1}{u_{1}}\leq\frac{1}{u_{2}}\left(\frac{\alpha}{n}+1\right)$.
Consequently we get the opposite for $R>1$:
\begin{eqnarray*}
 &  & \sup_{B_{R}\left(0\right),R>1}\left(\frac{\left|B\right|}{\omega_{\alpha,\beta}^{-1}\left(B\right)}\right)^{1/p}\left(\frac{\omega_{\alpha,\beta}\left(B\right)}{\left|B\right|}\right)^{1/u_{2}-1/p}\left|B\right|^{1/u_{2}-1/u_{1}}\\
 & \sim & \sup_{B_{R}\left(0\right),R>1}\left|B\right|^{\frac{1}{u_{2}}\left(\frac{\beta}{n}+1\right)-\frac{1}{u_{2}}}<\infty,
\end{eqnarray*}
which is true for $\frac{1}{u_{2}}\left(1+\frac{\beta}{n}\right)\leq\frac{1}{u_{1}}$.
Roughly speaking one could say $\left|B\right|/\omega_{\alpha,\beta}^{-1}\left(B\right)$
is the counterpart to $\left(\omega_{\alpha,\beta}\left(B\right)/\left|B\right|\right)^{-1}$.
The remaining cases of the proof run analogously to the proof of Corollary
\ref{cor: ExampleWeight1} and do not yield any new information about
upper or lower boundaries for $u_{1}$. Note however, that the maximum
term in \eqref{eq:MaximumCaseForBalls} $\max\left(\alpha,0\right)$
turns into $\min\left(\alpha,0\right)$ and we are left with \eqref{eq:ZweitesBeispielErsteBedingung}.

(ii) As we have seen in Corollary \ref{cor: ExampleWeight1}, we know
that $\omega_{\alpha,\beta}\in A_{r}$ if and only if the parameters
satisfy $-n<\alpha,\beta<n\left(r-1\right)$. Again we use $\omega_{\alpha,\beta}^{-1}=\omega_{-\alpha,-\beta}$.
Rearranging the inequality yields $r_{\omega_{\alpha,\beta}^{-1}}=1+\frac{\max\left(-\alpha,-\beta\right)}{n}=1-\frac{\min\left(\alpha,\beta\right)}{n}$.
Now using the second condition of Proposition \ref{prop:UngewichtInGewicht},
we see that we require $r_{\omega_{\alpha,\beta}^{-1}}<2-\frac{p}{q}$.
Inserting the preceding equality and rearranging the terms yields
$\frac{1}{q}<\frac{1}{p}\left(1+\frac{\min\left(\alpha,\beta\right)}{n}\right)$.
We also notice that $\min\left(\alpha,\beta,0\right)=\min\left(\alpha,\beta\right)$
since \eqref{eq:ZweitesBeispielErsteBedingung} implies $\beta\leq0$.
\end{proof}
\begin{cor}
\label{cor:CorollaryConclusion}Let $\min\left(\alpha_{1},\beta_{1}\right)>-n$
and $\left|\alpha_{2}\right|,\left|\beta_{2}\right|<n$. Furthermore
let the parameters satisfy $0<p_{1}\leq u_{1}$ and $0<p_{2}\leq u_{2}$.
Assume
\begin{eqnarray}
 &  & \max\left(\frac{1}{u_{2}}\left(1+\frac{\beta_{2}}{n}\right),\frac{1}{u_{1}}\left(1+\frac{\max\left(\alpha_{1}\right)}{n}\right)\right)\nonumber \\
 & \le & \min\left(\frac{1}{u_{2}}\left(1+\frac{\min\left(\alpha_{2},0\right)}{n}\right),\frac{1}{u_{1}}\left(1+\frac{\beta_{1}}{n}\right)\right)\mbox{\,\ and}\label{eq:ExampleConclusion1}
\end{eqnarray}
\begin{eqnarray}
 &  & \frac{1}{p_{1}}\left(1+\frac{\beta_{1}}{n}\right)=\frac{1}{p_{1}}\left(1+\frac{\max\left(\alpha_{1},\beta_{1}\right)}{n}\right).\label{eq:ExampleConclusion2}\\
 & < & \frac{1}{p_{2}}\left(1+\frac{\min\left(\alpha_{2},\beta_{2}\right)}{n}\right)=\frac{1}{p_{2}}\left(1+\frac{\beta_{2}}{n}\right)\nonumber 
\end{eqnarray}
Then the following embedding holds
\[
M_{u_{1},p_{1}}\left(\omega_{\alpha_{1},\beta_{1}}\right)\hookrightarrow M_{u_{2},p_{2}}\left(\omega_{\alpha_{2},\beta_{2}}\right).
\]
\end{cor}

\begin{proof}
The proof is essentially a combination of Corollary \ref{cor: ExampleWeight1}
and Corollary \ref{cor:ExampleWeight2}. We want to find parameters
$u$ and $p$ such that
\[
M_{u_{1},p_{1}}\left(\omega_{\alpha_{1},\beta_{1}}\right)\overset{\mbox{\tiny(i)}}{\hookrightarrow}M_{u,p}\overset{\mbox{\tiny(ii)}}{\hookrightarrow}M_{u_{2},p_{2}}\left(\omega_{\alpha_{2},\beta_{2}}\right).
\]
We use Corollary \ref{cor: ExampleWeight1} in (i) such that \eqref{eq:ErstesBeispielErsteBedingung}
and \eqref{eq:ErstesBeispielZweiteBedingung} hold and Corollary \ref{cor:ExampleWeight2}
in (ii) such that \eqref{eq:ZweitesBeispielErsteBedingung} and \eqref{eq:ZweitesBeispielZweiteBedingung}
hold (where we have to be careful about the correct use of parameters).
Combining the first conditions of Corollary \ref{cor: ExampleWeight1}
and Corollary \ref{cor:ExampleWeight2}, i.e. \eqref{eq:ErstesBeispielErsteBedingung}
and \eqref{eq:ZweitesBeispielErsteBedingung} yields \eqref{eq:ExampleConclusion1}.
Note that the conditions $\beta_{1}=\max\left(\alpha_{1},\beta_{1}\right)$
and $\beta_{2}=\min\left(\alpha_{2},\beta_{2}\right)$ are implicitly
taken from there. Combining the second condition of Corollary \ref{cor: ExampleWeight1}
and Corollary \ref{cor:ExampleWeight2}, i.e. \eqref{eq:ErstesBeispielZweiteBedingung}
and \eqref{eq:ZweitesBeispielZweiteBedingung} yields \eqref{eq:ExampleConclusion2}. 

Also note that the condition $0<p<u$ for the interposed space $M_{u,p}$
is naturally fulfilled if the above criteria are satisfied. We have
\begin{eqnarray*}
\frac{1}{u} & \overset{\tiny\mbox{\eqref{eq:ExampleConclusion1}}}{\leq} & \min\left(\frac{1}{u_{2}}\left(1+\frac{\min\left(\alpha_{2},0\right)}{n}\right),\frac{1}{u_{1}}\left(1+\frac{\beta_{1}}{n}\right)\right)\\
 & \leq & \frac{1}{u_{1}}\left(1+\frac{\beta_{1}}{n}\right)\\
 & \leq & \frac{1}{p_{1}}\left(1+\frac{\max\left(\alpha_{1},\beta_{1}\right)}{n}\right)\\
 & \overset{\tiny\eqref{eq:ExampleConclusion2}}{<} & \frac{1}{p}.
\end{eqnarray*}
\end{proof}
\begin{example}
\label{ExampleEnd}As an example for the preceding Corollary \ref{cor:CorollaryConclusion}
we investigate $\omega_{\alpha_{1},\beta_{1}}$ and $\omega_{\alpha_{2},\beta_{2}}$
with $\alpha_{1}=\alpha_{2}=:\alpha$ and $\beta_{1}=\beta_{2}=:\beta$,
where $\left|\alpha\right|,\left|\beta\right|<n$. Now since \eqref{eq:ErstesBeispielErsteBedingung}
implies $\beta\geq0$, \eqref{eq:ZweitesBeispielErsteBedingung} implies
$\beta\leq0$ and \eqref{eq:ExampleConclusion1} is a combination
of both, the only possible solution is $\beta=0$. Now \eqref{eq:ExampleConclusion1}
turns into

\[
\max\left(\frac{1}{u_{2}},\frac{1}{u_{1}}\left(1+\frac{\max\left(\alpha,0\right)}{n}\right)\right)\leq\min\left(\frac{1}{u_{2}}\left(1+\frac{\min\left(\alpha,0\right)}{n}\right),\frac{1}{u_{1}}\right),
\]

which implies $\alpha=0$ and further $u_{1}=u_{2}$. Lastly this
also implies $p_{2}<p_{1}$, meaning

\[
M_{u_{1},p_{1}}\left(\omega_{\alpha,\beta}\right)\hookrightarrow M_{u_{2},p_{2}}\left(\omega_{\alpha,\beta}\right)
\]
\[
\mbox{if }\,M_{u_{1},p_{1}}\hookrightarrow M_{u_{2},p_{2}}.
\]

We have already seen a similar result in Corollary \ref{cor:Sch=0000F6nesKorollar}
for unweighted Morrey spaces and in a more general sense in Corollary
\ref{cor:MuckenhouptH=0000F6lder} or Corollary \ref{cor:MuckenhouptH=0000F6lder2}
respectively.
\end{example}

The following theorem can be considered the main result of this thesis.
It gives sufficient conditions for an embedding between two weighted
Morrey spaces equipped with different weights belonging to different
Muckenhoupt classes. Because of the missing necessary condition, it
is not as strong as Theorem \ref{thm:Sch=0000F6nesTheorem} or Corollary
\ref{cor:Sch=0000F6nesKorollar} respectively, but manages to give
insight into some possibilities. 
\begin{thm}
\label{thm:GerholdsSatz}Let $0<q<u_{1}$ and $0<p<u_{2}$. Furthermore
let $\nu^{-1}\in A_{\infty}$ with $r_{\nu^{-1}}<2$. Assume
\begin{equation}
\sup_{B}\frac{\nu\left(B\right)^{1/u_{2}}}{\omega\left(B\right)^{1/u_{1}}}\cdot\left(\frac{\left|B\right|}{\nu\left(B\right)}\right)^{1/p}\left(\frac{\left|B\right|}{\nu^{-1}\left(B\right)}\right)^{1/p}<\infty\label{eq:ConditionGerhold1}
\end{equation}
\begin{equation}
\frac{r_{\omega}}{2-r_{\nu^{-1}}}<\frac{q}{p}.\label{eq:ConditionGerhold2}
\end{equation}
Then the following embedding holds
\[
M_{u_{1},q}\left(\omega\right)\hookrightarrow M_{u_{2},p}\left(\nu\right).
\]
Furthermore, if $\omega,\nu^{-1}\in A_{1}$ then $p=q$ is admitted.
\end{thm}

\begin{proof}
Let $f\in M_{u_{1},q}\left(\omega\right)$ and $B$ be ball. We begin
by using Hölder's inequality twice; the first time for $\frac{q}{p}>1$
and the second time for $s>1$ to be chosen later according to its
use: 
\begin{eqnarray*}
 &  & \left(\int_{B}\left|f\left(x\right)\right|^{p}\nu\left(x\right)dx\right)^{1/p}\\
 & = & \left(\int_{B}\left|f\left(x\right)\right|^{p}\nu\left(x\right)\omega\left(x\right)^{\frac{p}{q}}\omega\left(x\right)^{-\frac{p}{q}}dx\right)^{1/p}\\
 & \leq & \left(\int_{B}\left|f\left(x\right)\right|^{q}\omega\left(x\right)dx\right)^{1/q}\left(\int_{B}\omega\left(x\right)^{-\frac{p}{q-p}}\nu\left(x\right)^{\frac{q}{q-p}}dx\right)^{1/p-1/q}\\
 & \leq & \left(\int_{B}\left|f\left(x\right)\right|^{q}\omega\left(x\right)dx\right)^{1/q}\left(\int_{B}\omega\left(x\right)^{-\frac{sp}{q-p}}dx\right)^{\frac{1}{s}\left(\frac{1}{p}-\frac{1}{q}\right)}\left(\int_{B}\nu\left(x\right)^{\frac{s'q}{q-p}}dx\right)^{\frac{1}{s'}\left(\frac{1}{p}-\frac{1}{q}\right)}
\end{eqnarray*}
Multiplying the inequality with $\nu\left(B\right)^{\frac{1}{u_{2}}-\frac{1}{p}}$,
inserting $\omega\left(B\right)^{\frac{1}{u_{1}}-\frac{1}{q}}\omega\left(B\right)^{\frac{1}{q}-\frac{1}{u_{1}}}$
and taking the supremum over all balls $B$ yields
\begin{eqnarray}
\Vert f\vert M_{u_{2},p}\left(\nu\right)\Vert & \leq & \Vert f\vert M_{u_{1},q}\left(\omega\right)\Vert\sup_{B}\nu\left(B\right)^{1/u_{2}-1/p}\omega\left(B\right)^{1/q-1/u_{1}}\nonumber \\
 &  & \cdot\left(\int_{B}\omega\left(x\right)^{-\frac{sp}{q-p}}dx\right)^{\frac{1}{s}\left(\frac{1}{p}-\frac{1}{q}\right)}\left(\int_{B}\nu\left(x\right)^{\frac{s'q}{q-p}}dx\right)^{\frac{1}{s'}\left(\frac{1}{p}-\frac{1}{q}\right)}.\label{eq:SupremumsTermImGerholdschenSatz}
\end{eqnarray}

Now we will take care of the terms within the supremum term individually.
As we have done before in Proposition \ref{prop:GewichtInUngewicht}
and \ref{prop:UngewichtInGewicht} we want to use properties of the
Muckenhoupt weights $\omega$ and $\nu^{-1}$. To do so, we choose
the free parameter $s>1$ accordingly. We have two conditions:
\begin{equation}
\mbox{(i)\,\ }\frac{q-p}{sp}+1=:r>r_{\omega}\,\,\,\,\mbox{ and }\,\,\,\,\mbox{(ii) }\frac{q-p}{s'q}+1=:\varrho>r_{\nu^{-1}},\label{eq:BothInequalities}
\end{equation}
and since $s'=\frac{s}{s-1}$, we can rearrange both inequalities,
acquiring
\[
\frac{q-p}{\left(r_{\omega}-1\right)p}>s>\frac{q-p}{\left[\left(2-r_{\nu^{-1}}\right)q\right]-p}.
\]
This leaves a certain span for $s$ to fulfill both of the above conditions,
if only
\[
\frac{q}{p}>\frac{r_{\omega}}{\left(2-r_{\nu^{-1}}\right)},
\]
and $2>r_{\nu^{-1}}>1$, which we find the only admissible case when
rearranging both inequalities in \eqref{eq:BothInequalities} (Note
that $s>1$ is also fulfilled). Now this gives the second condition
of the proposition (i.e. \eqref{eq:ConditionGerhold2}). 

The next step is to use the Muckenhoupt property of $\omega$ and
$\nu^{-1}$ for the integral terms including $\omega$ and $\nu$
accordingly, as we have done before in Proposition \ref{prop:GewichtInUngewicht}
and \ref{prop:UngewichtInGewicht}. This yields 
\begin{eqnarray*}
 &  & \omega\left(B\right)^{1/q-1/u_{1}}\left(\int_{B}\omega\left(x\right)^{-\frac{sp}{q-p}}dx\right)^{\frac{1}{s}\left(\frac{1}{p}-\frac{1}{q}\right)}\\
 & \leq & A\omega\left(B\right)^{-1/u_{1}}\left|B\right|^{\frac{1}{s}\left(\frac{1}{p}-\frac{1}{q}\right)}\left|B\right|^{\frac{1}{q}}.
\end{eqnarray*}
At last we deal with the integral term including $\nu$. Similar to
Proposition \ref{prop:UngewichtInGewicht} we write $\nu^{\frac{q}{q-p}}=\left(\nu^{-1}\right)^{-\frac{q}{q-p}}$
to use the Muckenhoupt property. We have 
\begin{eqnarray*}
 &  & \nu\left(B\right)^{1/u_{2}-1/p}\left(\int_{B}\left(\nu\left(x\right)^{-1}\right)^{-\frac{s'q}{q-p}}dx\right)^{\left(\frac{q-p}{s'q}\right)\frac{1}{p}}\\
 & \leq & C\nu\left(B\right)^{1/u_{2}-1/p}\left(\frac{\left|B\right|}{\nu^{-1}\left(B\right)}\right)^{1/p}\left|B\right|^{\frac{1}{s'}\left(\frac{1}{p}-\frac{1}{q}\right)},
\end{eqnarray*}
and thus \eqref{eq:SupremumsTermImGerholdschenSatz} can finally be
estimated as
\begin{eqnarray*}
 &  & \nu\left(B\right)^{1/u_{2}-1/p}\omega\left(B\right)^{1/q-1/u_{1}}\left(\int_{B}\omega\left(x\right)^{-\frac{sp}{q-p}}dx\right)^{\frac{1}{s}\left(\frac{1}{p}-\frac{1}{q}\right)}\left(\int_{B}\nu\left(x\right)^{\frac{s'q}{q-p}}dx\right)^{\frac{1}{s'}\left(\frac{1}{p}-\frac{1}{q}\right)}\\
 & \leq & C\omega\left(B\right)^{-\frac{1}{u_{1}}}\nu\left(B\right)^{\frac{1}{u_{2}}-\frac{1}{p}}\left(\frac{\left|B\right|}{\nu^{-1}\left(B\right)}\right)^{1/p}\left|B\right|^{\left(\frac{1}{s'}+\frac{1}{s}\right)\left(\frac{1}{p}-\frac{1}{q}\right)+\frac{1}{q}}\\
 & =C & \frac{\nu\left(B\right)^{1/u_{2}}}{\omega\left(B\right)^{1/u_{1}}}\left(\frac{\left|B\right|}{\nu\left(B\right)}\right)^{1/p}\left(\frac{\left|B\right|}{\nu^{-1}\left(B\right)}\right)^{1/p},
\end{eqnarray*}
where the supremum of the term is finite by assumption \eqref{eq:ConditionGerhold1}.
This concludes the first part of the proof.

Let us now have a look at the case $\omega,\nu^{-1}\in A_{1}$. In
that case $p=q$ is admitted, and we get the result by combining the
arguments for $p=q$ used in Proposition \ref{prop:GewichtInUngewicht}
and Proposition \ref{prop:UngewichtInGewicht} as follows; let $f\in M_{u_{1},q}\left(\omega\right)$
and $B$ be a ball, then
\begin{eqnarray}
 &  & \nu\left(B\right)^{1/u_{2}-1/p}\left(\int_{B}\left|f\left(x\right)\right|^{p}\nu\left(x\right)dx\right)^{1/p}\nonumber \\
 & \leq & \nu\left(B\right)^{1/u_{2}-1/p}\Vert\nu\vert L_{\infty}\left(B\right)\Vert\Vert\omega^{-1}\vert L_{\infty}\left(B\right)\Vert\left(\int_{B}\left|f\left(x\right)\right|^{q}\omega\left(x\right)dx\right)^{1/q}\nonumber \\
 & = & \nu\left(B\right)^{1/u_{2}-1/p}\Vert\nu\vert L_{\infty}\left(B\right)\Vert\Vert\omega^{-1}\vert L_{\infty}\left(B\right)\Vert\omega\left(B\right)^{1/p-1/u_{1}}\nonumber \\
 &  & \cdot\omega\left(B\right)^{1/u_{1}-1/q}\left(\int_{B}\left|f\left(x\right)\right|^{1/q}\omega\left(x\right)dx\right)^{1/q}.\label{eq:equationnnnn}
\end{eqnarray}
As we have done before in both propositions, we use the $A_{1}$ Muckenhoupt
properties of $\omega$ and $\nu^{-1}$ and take the supremum over
all balls $B$ in \eqref{eq:equationnnnn}, which yields
\begin{eqnarray*}
 &  & \sup_{B}\nu\left(B\right)^{1/u_{2}-1/p}\Vert\nu\vert L_{\infty}\left(B\right)\Vert\Vert\omega^{-1}\vert L_{\infty}\left(B\right)\Vert\omega\left(B\right)^{1/p-1/u_{1}}\Vert f\vert M_{u_{1},q}\left(\omega\right)\Vert\\
 & \leq & \sup_{B}\nu\left(B\right)^{1/u_{2}-1/p}\left|B\right|^{2/p}\nu^{-1}\left(B\right)^{-1/p}\omega\left(B\right)^{-1/p}\omega\left(B\right)^{1/p-1/u_{1}}\Vert f\vert M_{u_{1},q}\left(\omega\right)\Vert\\
 & = & \sup_{B}\left(\frac{\nu\left(B\right)^{1/u_{2}}}{\omega\left(B\right)^{1/u_{1}}}\right)\left(\frac{\left|B\right|}{\nu\left(B\right)}\right)^{1/p}\left(\frac{\left|B\right|}{\nu^{-1}\left(B\right)}\right)^{1/p}\Vert f\vert M_{u_{1},q}\left(\omega\right)\Vert.
\end{eqnarray*}
Again, this is finite by assumption \eqref{eq:ConditionGerhold1}
and thus $f\in M_{u_{2},p}\left(B\right)$, which concludes the proof.
\end{proof}
\begin{rem}
In case of $\nu\equiv\nu^{-1}\equiv1$ we see that $\nu^{-1}\in A_{1}$
and thus $r_{\nu^{-1}}=1$. Consequently \eqref{eq:ConditionGerhold1}
turns into \eqref{eq:GewichtInUngewichtErsteBedingung} and \eqref{eq:ConditionGerhold2}
replicates \eqref{eq:GewichtInUngewichtZweiteBedingung}. Equivalently
for $\omega\equiv1$ we see $r_{\omega}=1$ and \eqref{eq:ConditionGerhold1}
turns into \eqref{eq:UngewichtInGewichtErsteBedingung} and \eqref{eq:ConditionGerhold2}
reproduces \eqref{eq:UngewichtInGewichtZweiteBedingung}. Lastly if
$\omega\equiv\nu\equiv\nu^{-1}\equiv1$, we see that $\omega,\nu,\nu^{-1}\in A_{1}$
and hence $r_{\omega}=r_{\nu^{-1}}=1$. With this assignment of parameters
Theorem \ref{thm:GerholdsSatz} becomes Corollary \ref{cor:MuckenhouptH=0000F6lder}
or Corollary \ref{cor:MuckenhouptH=0000F6lder2} respectively, since
\eqref{eq:ConditionGerhold1} becomes $\sup_{B}\left|B\right|^{1/u_{2}-1/u_{1}}$
which can only be finite if $u_{1}=u_{2}$ and \eqref{eq:ConditionGerhold2}
turns into $\frac{q}{p}\geq1$, which also seems natural.
\end{rem}

\begin{cor}
\label{cor:EndCorollary}Let $\min\left(\alpha_{1},\beta_{1}\right)>-n$
and $\left|\alpha_{2}\right|,\left|\beta_{2}\right|<n$. Also let
$0<q\leq u_{1}$ and $0<p\leq u_{2}$. Assume 
\begin{equation}
\frac{1}{2n}\left(\frac{\beta_{1}}{u_{1}}-\frac{\beta_{2}}{u_{2}}\right)\geq\frac{1}{u_{2}}-\frac{1}{u_{1}}\geq\frac{1}{2n}\max\left(\frac{\alpha_{1}}{u_{1}}-\frac{\alpha_{2}}{u_{2}},0\right)\mbox{ and}\label{eq:EndKorollarBedingung1}
\end{equation}
\begin{equation}
\frac{q}{p}>\frac{n+\max\left(\alpha_{1},\beta_{1},0\right)}{n+\min\left(\alpha_{2},\beta_{2},0\right)}.\label{eq:EndKorollarBedingung2}
\end{equation}
Then the following embedding holds
\[
M_{u_{1},q}\left(\omega_{\alpha_{1},\beta_{1}}\right)\hookrightarrow M_{u_{2},p}\left(\omega_{\alpha_{2},\beta_{2}}\right).
\]

\newpage{}
\end{cor}

\begin{proof}
The proof is again a combination of previous corollaries. \eqref{eq:EndKorollarBedingung2}
is a direct conclusion of \eqref{eq:ExampleConclusion2}, which in
itself was a conclusion of \eqref{eq:ErstesBeispielZweiteBedingung}
and \eqref{eq:ZweitesBeispielZweiteBedingung}. Now \eqref{eq:EndKorollarBedingung1}
is a deduction of \eqref{eq:ErstesBeispielErsteBedingung} and \eqref{eq:ZweitesBeispielErsteBedingung}.
We simply subtract the latter from the first condition, rearrange
terms and take into consideration that $\frac{\max\left(\alpha_{1},0\right)}{u_{1}}-\frac{\min\left(\alpha_{2},0\right)}{u_{2}}\geq\max\left(\frac{\alpha_{1}}{u_{1}}-\frac{\alpha_{2}}{u_{2}}\right)$
for the configuration of parameters given above.
\end{proof}
\begin{rem}
An immediate consequence of Corollary \ref{cor:EndCorollary} in case
of $\alpha_{1}=\alpha_{2}=\alpha$ and $\beta_{1}=\beta_{2}=\beta$
is
\[
M_{u_{1},q}\left(\omega_{\alpha,\beta}\right)\hookrightarrow M_{u_{2},p}\left(\omega_{\alpha,\beta}\right)
\]
whenever $u_{1}=u_{2}$ and $\frac{1}{p}>\frac{1}{q}\frac{n+\max\left(\alpha_{1},\beta_{1},0\right)}{n+\min\left(\alpha_{2},\beta_{2},0\right)}>\frac{1}{q}$,
as we have seen before in Example \ref{ExampleEnd}.
\end{rem}

Except for Corollary \ref{cor:MuckenhouptH=0000F6lder} and Corollary
\ref{cor:MuckenhouptH=0000F6lder2} we only used weights belonging
to Muckenhoupt classes throughout this section. To conclude this section
and our study concerning the topic within this thesis, we give another
result for general weights, which can be deduced rather easily and
does not need as much preliminaries.
\begin{lem}
Let $0<p\leq u<\infty$ and $\omega$, $\nu$ be weights. We have
\[
M_{u,p}\left(\omega\right)=M_{u,p}\left(\nu\right)\Longleftrightarrow\nu\sim\omega\mbox{ a.e.}.
\]
\end{lem}

\begin{proof}
For the first part, let $M_{u,p}\left(\omega\right)=M_{u,p}\left(\nu\right)$.
By the equivalence of norms, we have
\begin{eqnarray*}
 &  & c\nu\left(B\right)^{\frac{1}{u}-\frac{1}{p}}\int_{B}\left|f\left(x\right)\right|^{p}\nu\left(x\right)dx\\
 & \leq & \omega\left(B\right)^{\frac{1}{u}-\frac{1}{p}}\int_{B}\left|f\left(x\right)\right|^{p}\omega\left(x\right)dx\\
 & \leq & C\nu\left(B\right)^{\frac{1}{u}-\frac{1}{p}}\int_{B}\left|f\left(x\right)\right|^{p}\nu\left(x\right)dx,
\end{eqnarray*}
which quickly leads to $c\nu\leq\omega\leq C\nu$. The second part
works analogue.
\end{proof}
\newpage{} 

\end{document}